\documentclass[12pt,reqno]{amsart}

\usepackage[a4paper,margin=1in]{geometry} 

\usepackage{amscd}
\usepackage{amssymb}
\usepackage{mathrsfs}
\usepackage{amsmath}
\usepackage{mathtools}
\usepackage{mathabx}
\usepackage{amsthm}
\usepackage[all,cmtip]{xy}
\usepackage{enumerate}
\usepackage{enumitem}
\usepackage{eucal}
\usepackage{ytableau}

\usepackage{color}
\usepackage{wasysym}

\usepackage{pifont}
\usepackage{graphicx}
\newcommand{\dualstar}{\mathord{\text{\scalebox{0.85}{\ding{73}}}}}
\newcommand{\medstar}{\mathord{\scalebox{1}{$\star$}}}  
\newcommand{\biggstar}{\mathord{\scalebox{1.2}{$\star$}}}

\allowdisplaybreaks

\usepackage[plainpages,urlcolor=red,linktocpage=true]{hyperref}

\numberwithin{equation}{section}

\newcommand{\thmref}[1]{Theorem~\ref{#1}}
\newcommand{\secref}[1]{Section~\ref{#1}}

\newcommand{\lemref}[1]{Lemma~\ref{#1}}
\newcommand{\propref}[1]{Proposition~\ref{#1}}

\newcommand{\eqnref}[1]{(\ref{#1})}

\newtheorem{theorem}{Theorem}[section]
\newtheorem{lemma}[theorem]{Lemma}
\newtheorem{proposition}[theorem]{Proposition}
\newtheorem{corollary}[theorem]{Corollary}

\theoremstyle{definition}

\newcommand{\Dpqn}{D^+_{p,q|n}}
\newcommand{\Ppqn}{P^+_{p,q|n}}
\newcommand{\Ppqnk}{\mathcal{P}^k_{p,q|n}}
\newcommand{\Ppqnd}{\mathcal{P}^d_{p,q|n}}
\newcommand{\Cpqnd}{\mathbb{C}^d_{p,q|n}}
\newcommand{\Dpqnd}{\mathbb{D}^d_{p,q|n}}

\theoremstyle{remark}
\newtheorem{remark}[theorem]{Remark}

\newcommand{\RNum}[1]{\uppercase\expandafter{\romannumeral #1\relax}}

\begin{document}

\title[Classification of unitary modules]{Classification of\\[.1cm] 
irreducible unitary modules over $\mathfrak{u}(p,q|n)$}
\author{Mark D.~Gould, Artem Pulemotov, J\o rgen Rasmussen, Yang Zhang}

\address{School of Mathematics and Physics, The University of Queensland, St Lucia, QLD 4072, Australia}

\email{m.gould1\;\!@\;\!uq.edu.au}
\email{a.pulemotov\;\!@\;\!uq.edu.au}
\email{j.rasmussen\;\!@\;\!uq.edu.au}
\email{yang.zhang\;\!@\;\!uq.edu.au}

\begin{abstract}
We classify all irreducible highest-weight unitary modules over the non-compact real form $\mathfrak{u}(p,q|n)$ of the general linear Lie superalgebra $\mathfrak{gl}_{p+q|n}$. The classification is given by explicit necessary and sufficient conditions on the highest weights, and our approach combines the Howe duality for $\mathfrak{gl}_{p+q|n}$ with a quadratic invariant of the maximal compact subalgebra. Using this classification result, we also classify all irreducible lowest-weight unitary modules over $\mathfrak{u}(p,q|n)$ via duality, and all irreducible unitary modules over $\mathfrak{u}(n|q,p)$ via an isomorphism of Lie superalgebras. 
\end{abstract}

\maketitle

\tableofcontents

\section{Introduction}

The theory of Lie superalgebras \cite{Kac77a} and their representations plays a fundamental role in the understanding and exploitation of 
supersymmetry in physical systems. The notion of supersymmetry first arose in elementary particle physics and quantum field theory
but has since found applications in a variety of areas, including nuclear physics, integrable models, 
and string theory \cite{Jun96,DF04,Din07}. 

The representation theory of the simple basic classical Lie superalgebras was first investigated by Kac \cite{Kac77b},
who introduced the now familiar dichotomy between typical and atypical finite-dimensional irreducible representations. 
In this paper, we are concerned with the \textit{unitary} representations of the basic classical Lie superalgebra $\mathfrak{u}(p,q|n)$,
and as the corresponding modules admit contravariant positive-definite Hermitian forms, they are amenable to physical applications
where unitarity is a basic requirement.

Unitary representations of Lie superalgebras are natural generalisations of those appearing in the theory of ordinary Lie algebras,
and arise naturally from so-called star-operations. 
Such operations were first introduced in the super setting by Scheunert, Nahm and Rittenberg \cite{SNR77}, 
who showed that a basic classical Lie superalgebra admits at most two types of finite-dimensional unitary irreducible representations. 
These were subsequently classified by Gould and Zhang \cite{GZ90b}. 
Crucially, such unitary representations arise from star-operations corresponding to a compact real form of the underlying 
even subalgebra, and only exist for the type-I basic classical Lie superalgebras $\mathfrak{osp}_{2|2n}$ and~$\mathfrak{gl}_{m|n}$.

Apart from any physical applications, irreducible representations of Lie superalgebras are mathematically interesting in their own right. 
In particular, the category of finite-dimensional unitary modules (of a given type) is closed under tensor products, 
so the tensor product of two such modules is completely reducible. 
Moreover, there are two distinct types of unitary representations, and they are related by duality \cite{GZ90b}. 
The first type includes the so-called covariant tensor representations, while the second includes the contravariant tensor 
representations. This explains the applicability of Young diagram methods to this class of modules.
In fact, there exists a much larger class of non-tensorial typical unitary modules. 
Indeed, corresponding to every irreducible tensor module, there is a one-parameter family of typical unitary modules which are non-tensorial. 
It is particularly interesting that such representations underlie integrable electron models (and corresponding link polynomials),  
where the parameter labelling the modules has physical significance~\cite{GHLZ96}.

Unlike the finite-dimensional case, \textit{infinite-dimensional} unitary highest-weight representations exist for all basic 
classical Lie superalgebras. 
Here, we provide a classification of such $\mathfrak{u}(p,q|n)$-representations that arise from a star-operation for 
which $\mathfrak{u}(p,q)\oplus \mathfrak{u}(n)$ is the 
real form of the even complex subalgebra $\mathfrak{gl}_{p+q}\oplus \mathfrak{gl}_{n}$ of $\mathfrak{gl}_{p+q|n}$.

As noted above, the compact case ($q=0$) was treated in \cite{GZ90b}. In the non-compact case ($p,q\neq 0$), 
positive-energy unitary irreducible representations of $\mathfrak{su}(p,q|n)$ have been studied for small $n$ and for $p=q=2$, 
motivated by superconformal field theory; see \cite{GZ90a,GV19} and references therein. Furutsu and Nishiyama \cite{FN91} 
subsequently classified the highest-weight unitary irreducible $\mathfrak{su}(p,q|n)$-modules with integer highest weights. 
By design, this classification does not cover the large class of unitary modules with non-integer highest weights.
More general and systematic classification results for highest-weight unitary irreducible $\mathfrak{su}(p,q|n)$-modules were later 
proposed by Jakobsen \cite{Jak94} and, more recently, by G\"unaydin and Volin \cite{GV19}. 
As noted in \cite{GV19}, Jakobsen's classification does not include all positive-energy unitary representations of~$\mathfrak{su}(2,2|n)$.

In the present paper, we extend the approach of \cite{GZ90b} and classify both highest-weight and lowest-weight 
unitary $\mathfrak{u}(p,q|n)$-modules. 
The latter are related to the former by duality, a fact that, to the best of our knowledge, has not been explicitly noted in the literature. 
We give detailed proofs of both necessity and sufficiency of the classification, based on an induced module construction and 
an application of Howe duality to the super setting \cite{CW01,CLZ04}.

A key feature of our approach is that the classification is formulated directly with respect to the \emph{standard} Borel subalgebra,
enhancing the applicability in both physics and mathematics.
This also distinguishes our work from that of G\"unaydin and Volin \cite{GV19}, who obtained unitarity conditions  using Young diagrams and 
oscillator techniques for a \emph{non-standard} Borel subalgebra. Their classification was recently recovered by Schmidt \cite{Sch26}
using an algebraic Dirac operator and corresponding Dirac inequalities, but still formulated in the non-standard setting.\footnote{The authors only became aware of Schmidt's work just prior to submission of the present paper; all results reported here were obtained before the preprint \cite{Sch26} appeared.}
In addition, since $\mathfrak{u}(p,q|n)$ has a nontrivial centre, its unitary modules involve a twist parameter 
not present at the level of $\mathfrak{su}(p,q|n)$.

The paper is set up as follows. 
\secref{Sec: Star} specifies our notation and conventions and presents some general results in the area. 
An in-depth discussion of star-operations on $\mathfrak{gl}_{m|n}$ and the corresponding unitary modules is given in \secref{sec: glmn}.
\thmref{thm: main} in \secref{sec: class} presents our main result, with
Sections \ref{sec: inv}--\ref{sec: suff} devoted to the proof of this classification. 
Necessary conditions for unitarity are thus derived in \secref{sec: nec}, while a new criterion for unitarity is derived in \secref{sec: inv}. 
\secref{sec: Howe} outlines the relevant Howe duality in our setting.
Following a discussion of unitary modules with integral highest weights in \secref{sec: Integral},
the proof of sufficiency of the conditions derived in \secref{sec: nec} is given in \secref{sec: suff}.
This then completes the proof of the classification result in \thmref{thm: main}. In \secref{sec: lowest}, 
using \thmref{thm: main}, we classify the unitary lowest-weight $\mathfrak{u}(p,q|n)$-modules 
and unitary $\mathfrak{u}(n|q,p)$-modules. 
\medskip

\noindent\textbf{Convention.} We denote by $\mathbb{C}$ (respectively $\mathbb{R}$) the field of complex (respectively real) numbers, 
by $\mathbb{Z}_+$ (respectively $\mathbb{Z}_-$) the set of non-negative (respectively non-positive) integers, and by $\mathbb{Z}_2$ 
the ring of integers modulo 2.  We write $\otimes$ for the tensor product over $\mathbb{C}$, denote the imaginary unit 
by $\mathrm{i}\mkern1mu$, and  for $x\in \mathbb{R}$, write $\lfloor{x}\rfloor$ for the greatest integer less than or equal to $x$.
For any Lie superalgebra $\mathfrak{l}$, we denote by ${\rm U}(\mathfrak{l})$ the corresponding universal enveloping algebra.
When characterising modules and representations, we use the terms ``simple" (and ``semisimple") and ``irreducible" 
(and ``completely reducible") interchangeably.
Throughout, we work over $\mathbb{C}$, unless otherwise stated. 
All homomorphisms between super vector spaces are assumed to be even. 
Given a super vector space $V=V_{\bar{0}}\oplus V_{\bar{1}}$, 
the parity of homogeneous $v\in V$ is given by $[v]=\bar{0}$ (respectively $\bar{1}$) if $v\in V_{\bar{0}}$ (respectively~$V_{\bar{1}}$). 
If a super vector space $V$ is finite-dimensional, we let $V^*:=\mathrm{Hom}(V,\mathbb{C})$ denote the usual dual. 
If $V$ is infinite-dimensional, we assume that $V=\bigoplus_{r\in\mathbb{Z}} V_r$ is $\mathbb{Z}$-graded,
with each $V_r$ finite-dimensional, and by abuse of notation, we write $V^*$ for the graded dual space defined
as $V^*:=\bigoplus_{r\in\mathbb{Z}} V_r^*$.

\section{Star-operations and unitarity}
\label{Sec: Star}

\subsection{Basics and duals}

Largely following \cite{SNR77,GZ90b,CLZ04}, here we recall some basic facts about star-superalgebras and their unitary modules,
and how these carry over to Lie superalgebras. We also consider dual star-operations and the corresponding unitary modules.

First, a \textit{star-superalgebra} over $\mathbb{C}$ is an associative superalgebra $A=A_{\bar{0}}\oplus A_{\bar{1}}$ equipped 
with an even anti-linear anti-involution $\phi: A\to A$; i.e., 
$$
 \phi(ca)=\bar{c}a, \qquad
 \phi(ab)= \phi(b)\phi(a), \qquad 
 c\in \mathbb{C},\quad a,b\in A, 
$$
where $\bar{c}$ denotes the complex conjugate of $c$. A star-superalgebra homomorphism $f: (A, \phi)\to (A', \phi')$ is a 
superalgebra homomorphism satisfying $f\circ \phi= \phi'\circ f$.

Second, let $(A,\phi)$ be a star-superalgebra and $V$ a $\mathbb{Z}_2$-graded $A$-module. 
A Hermitian form $\langle -,- \rangle$ on $V$ is said to be \textit{positive-definite} if $\langle v, v \rangle > 0$ for all nonzero $v\in V$, 
and \textit{contravariant} if $\langle av, w \rangle = \langle v, \phi(a)w \rangle$ for all $a \in A$ and all $v, w \in V$. An $A$-module 
equipped with a positive-definite contravariant Hermitian form is called a \textit{unitary} $A$-module.

Third, a \textit{star-operation} on a complex Lie superalgebra $\mathfrak{g}$ is an even anti-linear map $\biggstar$ such that 
$$
 [X,Y]^{\medstar}= [Y^{\medstar},X^{\medstar}], \qquad 
 (X^{\medstar})^{\medstar}=X, \qquad 
 X,Y\in \mathfrak{g}.
$$
As this star-operation extends to an even anti-linear anti-involution on ${\rm U}(\mathfrak{g})$, 
it equips ${\rm U}(\mathfrak{g})$ with the structure of a star-superalgebra.
The star-induced notion of unitarity for star-superalgebras can thus be applied to ${\rm U}(\mathfrak{g})$-modules,
and hence to $\mathfrak{g}$-modules.

Fourth, a real Lie superalgebra $\mathfrak{u}$ is said to be a \textit{real form} of a complex Lie superalgebra $\mathfrak{g}$ 
if $\mathfrak{g}\cong \mathfrak{u}\otimes_{\,\mathbb{R}}\!\mathbb{C}$. A real form $\mathfrak{u}$ is called \textit{compact} 
if $\mathfrak{u}_{\bar{0}}$ is a compact Lie algebra; otherwise, it is called \textit{non-compact}. 
A real form can be induced from a star-operation on $\mathfrak{g}$, setting
$$
 \mathfrak{u}_{\bar{0}}={\rm span}_{\mathbb{R}}\{X\in \mathfrak{g}_{\bar{0}}\,|\, X^{\medstar}=-X\},\qquad
 \mathfrak{u}_{\bar{1}}={\rm span}_{\mathbb{R}}\{X\in \mathfrak{g}_{\bar{1}}\,|\, X^{\medstar}=\mathrm{i}\mkern1mu X\}.
$$

We now define the \textit{dual star-operation} $\dualstar$  by 
$$
 X^{\dualstar}:=(-1)^{[X]}X^{\medstar}, \qquad X\in \mathfrak{g}_{\bar{0}}\cup \mathfrak{g}_{\bar{1}},
$$
extended linearly to all of $\mathfrak{g}$.
As stated in \propref{prop: dualstar} below, the $\mathfrak{g}$-modules which are unitary with respect to $\biggstar$ 
are related by duality to those which are unitary with respect to $\dualstar$.

Indeed, let $V$ be a $\biggstar$-unitary $\mathfrak g$-module and recall our convention that, if $V$ is infinite-dimensional, 
then $V$ is assumed $\mathbb{Z}$-graded with finite-dimensional graded components, while $V^*$ denotes the graded dual. 
Let $\{v_i\}_{i\in I}$ be a homogeneous orthonormal $V$-basis with respect to the positive-definite Hermitian form $\langle -, -\rangle$,
and let $\{v_i^*\}_{i\in I}$ denote the dual basis of $V^*$, so that $v_i^*(v_j)=\delta_{ij}$ for all $i,j\in I$.
The Hermitian form on $V$ induces a positive-definite Hermitian form on $V^*$ such that $\{v_i^*\}_{i\in I}$ is an orthonormal basis. 
In particular, for any homogeneous $X\in\mathfrak g$ and all $i,j\in I$, we have
$$
 v_i^*(Xv_j)=\langle v_i,Xv_j\rangle.
$$
Moreover, the dual $\mathfrak g$-module structure on $V^*$ is given by
$$
 (Xv_i^*)(v_j)=-(-1)^{[X][v_i]}v_i^*(Xv_j).
$$
Since every $f\in V^*$ can be written as $f=\sum_{i\in I} f(v_i)\,v_i^*$, it follows that
\begin{align*}
 \langle v_i^*,Xv_j^*\rangle
 &=\sum_{k\in I}\,\langle v_i^*,v_k^*\rangle\,(Xv_j^*)(v_k)
 =-(-1)^{[X][v_j]}\sum_{k\in I}\,\langle v_i^*,v_k^*\rangle\, v_j^*(Xv_k)
 \\[.1cm]
 &=-(-1)^{[X][v_j]}\,v_j^*(Xv_i)
 =-(-1)^{[X][v_j]}\,\langle v_j,Xv_i\rangle.
\end{align*}
Using this together with the unitarity of $V$ and the conjugation symmetry of the Hermitian form, we obtain
\begin{align*}
 \langle v_i^{\ast}, X^{\dualstar} v_j^{\ast}\rangle
 &= \langle v_i^{\ast}, (-1)^{[X]}X^{\medstar} v_j^{\ast}\rangle
 =-(-1)^{[X]}(-1)^{[X][v_j]}\,\langle v_j, X^{\medstar}v_i\rangle 
 \\[.1cm]
 &=-(-1)^{[X]}(-1)^{[X][v_j]}\,\langle X v_j, v_i\rangle
 = -(-1)^{[X]}(-1)^{[X][v_j]}\,\overline{\langle v_i, X v_j\rangle}
 \\[.1cm] 
 &= -(-1)^{[X][v_i]}\,\overline{\langle v_i, X v_j\rangle}
 =  \overline{\langle v_j^{\ast}, X v_i^{\ast}\rangle}
 = \langle  X v_i^{\ast}, v_j^{\ast}\rangle.
\end{align*}
Note that this is independent of the choice of homogeneous orthonormal basis of $V$.
In conclusion, we have the following result.
\begin{proposition}\label{prop: dualstar}
If $V$ is a unitary $\mathfrak{g}$-module with respect to a star-operation, then the dual module $V^{\ast}$ is unitary 
with respect to the dual star-operation.
\end{proposition}

\subsection{Compatibility with the super Killing form}

Here, we note a compatibility between star-operations and the super Killing form on $\mathfrak{g}$. 
Although this will not be used in the remainder of the paper, it may be of independent interest. 

To set the stage, recall that the \textit{super Killing form} on $\mathfrak{g}$ is defined by
$$
 (X,Y):= \operatorname{Str}(\operatorname{ad}X\circ\operatorname{ad}Y), \qquad 
 X,Y\in \mathfrak{g}, 
$$
where $\operatorname{Str}$ denotes the supertrace and $\operatorname{ad}$ the adjoint representation of $\mathfrak{g}$.
We now let $\{x_a\}$ denote a homogeneous basis for $\mathfrak{g}$ and write
$$
 [x_a,x_b]=\sum_{c}\Gamma_{ab}^{c}\,x_c,
$$
where by super skew-symmetry, the \textit{structure constants} $\Gamma_{ab}^{c}$ satisfy
$$
 \Gamma_{ab}^{c}=-(-1)^{[a][b]}\Gamma_{ba}^{c}.
$$
Here, $[a]:=[x_a]$, and we note that $[a]+[b]=[c]$ for any index $c$ in $\Gamma^{c}_{ab}$.
It follows that
$$
 [x_a,[x_b,x_c]]=\sum_{d}\Gamma_{bc}^{d}[x_a,x_d]=\sum_{d,s}\Gamma_{ad}^{s}\Gamma_{bc}^{d}x_s,
$$
hence
$$
 (x_a,x_b)=\operatorname{Str}(\operatorname{ad}x_a\circ\operatorname{ad}x_b)=\sum_{c,d} (-1)^{[c]}\Gamma_{ad}^{c}\Gamma_{bc}^{d}.
$$
\begin{proposition}
Let $\biggstar$ be a star-operation on $\mathfrak{g}$. Then, for all $X,Y\in \mathfrak{g}_{\bar{0}}\cup \mathfrak{g}_{\bar{1}}$,
$$
 (X^{\medstar},Y^{\medstar})= (-1)^{[X]}\,\overline{(X,Y)}. 
$$
\end{proposition}
\begin{proof}
Using the basis and structure-constant notation from above, we have
$$
 [x_a,x_b]^{\medstar}
 =[x_b^{\medstar},x_a^{\medstar}],\qquad
 [x_a,x_b]^{\medstar}=\sum_{c}(\Gamma_{ab}^{c}x_c)^{\medstar},
$$
so
$$
 [x_a^{\medstar},x_b^{\medstar}]=\sum_c(\Gamma_{ba}^{c}x_c)^{\medstar}=\sum_c\overline{\Gamma_{ba}^{c}}\,x_c^{\medstar},
$$
hence
$$
 [x_a^{\medstar},[x_b^{\medstar},x_c^{\medstar}]]
 =\sum_{d}\overline{\Gamma_{cb}^{d}}[x_a^{\medstar},x_d^{\medstar}]
 =\sum_{d,s}\overline{\Gamma_{cb}^{d}}\,\overline{\Gamma_{da}^{s}}x_s^{\medstar}.
$$
It follows that 
\begin{align*}
 (x_a^{\medstar},x_b^{\medstar})
 &=\operatorname{Str}(\operatorname{ad}x_a^{\medstar}\circ \operatorname{ad}x_b^{\medstar})
 =\sum_{c,d} (-1)^{[c]}\,\overline{\Gamma_{cb}^{d}}\,\overline{\Gamma_{da}^{c}}
 =(-1)^{[a]}\sum_{c,d}(-1)^{[c]}\,\overline{\Gamma_{ad}^{c}}\,\overline{\Gamma_{bc}^{d}} \\[.1cm]
 &=(-1)^{[a]}\,\overline{(x_a,x_b)}.
\end{align*}
Extending by linearity, we obtain the desired result. 
\end{proof}

\section{The general linear Lie superalgebra}
\label{sec: glmn}

Here, we discuss the general linear Lie superalgebra $\mathfrak{gl}_{p+q|n}$, its real form $\mathfrak{u}(p,q|n)$, 
and its unitary modules.

\subsection{Algebraic structure}

Let $\{e_1,\ldots,e_{m+n}\}$ denote the standard (ordered) basis for the complex superspace $\mathbb{C}^{m|n}$
of dimension $m|n$. That is, $\{e_1,\ldots,e_m\}$
is an ordered basis for the even subspace $(\mathbb{C}^{m|n})_{\bar{0}}=\mathbb{C}^m$, while $\{e_{m+1},\ldots,e_{m+n}\}$
is an ordered basis for the odd subspace $(\mathbb{C}^{m|n})_{\bar{1}}=\mathbb{C}^n$.
Relative to the standard basis, the generators of the \textit{general linear Lie superalgebra} 
$\mathfrak{gl}_{m|n}:=\mathfrak{gl}(\mathbb{C}^{m|n})$ are $(m+n)\times (m+n)$ matrices:
$$
 \mathfrak{gl}_{m|n}
 = \Big\{\!\begin{pmatrix}
 A & B\\ 
 C & D
 \end{pmatrix} \Big|\, A\in \mathfrak{M}_{m,m},\, B\in \mathfrak{M}_{m, n},\, C\in \mathfrak{M}_{n,m},\, D\in \mathfrak{M}_{n, n} \Big\}, 
$$
where $\mathfrak{M}_{r,s}$, $r,s\in\mathbb{Z}_+$, denotes the complex space of $(r\times s)$-matrices. 
The even subalgebra $(\mathfrak{gl}_{m|n})_{\bar{0}}=\mathfrak{gl}_m \oplus \mathfrak{gl}_n$ consists of the matrices for which 
$B=0$ and $C=0$, while the odd subspace $(\mathfrak{gl}_{m|n})_{\bar{1}}$ consists of the matrices for which $A=0$ and $D=0$. 
The Lie bracket is defined for homogeneous $X,Y\in \mathfrak{gl}_{m|n}$ by $[X,Y]= XY-(-1)^{[X][Y]}YX$, extended bilinearly to all 
of $\mathfrak{gl}_{m|n}$. 

For each pair $a,b\in \{1, \dots, m+n\}$, let $E_{ab}$ denote the matrix unit of $\mathfrak{gl}_{m|n}$ such that, 
for all $c\in \{1, \dots, m+n\}$, $E_{ab}e_c= \delta_{b,c}e_a$, where $\delta_{b,c}$ is the Kronecker symbol. 
The set $\{E_{ab}\,|\,a,b=1,\ldots,m+n\}$ is then a basis for $\mathfrak{gl}_{m|n}$, and
$$
 [E_{ab},E_{cd}]=\delta_{b,c}E_{ad}-(-1)^{([a]+[b])([c]+[d])}\delta_{d,a}E_{cb}.
$$
A basis for the \textit{Cartan subalgebra} $\mathfrak{h}_{m|n}$ of $\mathfrak{gl}_{m|n}$ is $\{E_{aa}\,|\,a=1,\ldots,m+n \}$, 
while $\{E_{ab}\,|\, 1\leq a\leq b\leq m+n \}$ is a basis for the corresponding \textit{standard Borel subalgebra} $\mathfrak{b}_{m|n}$.

Let $\{\epsilon_a \,|\,a=1,\ldots,m+n\}$ be a basis for the dual space 
$\mathfrak{h}_{m|n}^{\ast}={\rm Hom}_{\mathbb{C}}(\mathfrak{h}_{m|n}, \mathbb{C})$ 
such that $\epsilon_a(E_{bb}) = \delta_{a,b}$ for all $a,b$. 
The set of positive roots relative to $\mathfrak{b}_{m|n}$ is
$$
 \Phi^+ = \{\epsilon_a - \epsilon_b \,|\, 1 \leq a < b \leq m+n\}.
$$
Writing $\delta_{\mu}= \epsilon_{m+\mu}$ for $\mu=1,\ldots,n$, the sets of even respectively odd positive roots are given by
\begin{align*}
 \Phi^+_{\bar{0}} &= \{\epsilon_i - \epsilon_j \,|\, 1 \leq i < j \leq m\} \cup \{\delta_{\mu} - \delta_{\nu} \,|\, 1 \leq \mu < \nu \leq n\},
 \\[.1cm]
 \Phi^+_{\bar{1}} &= \{\epsilon_i - \delta_{\mu} \,|\, 1 \leq i \leq m, 1\leq \mu\leq n\},
\end{align*}
while the simple root system is given by
$$
 \Delta= \{\epsilon_i-\epsilon_{i+1}, \epsilon_m-\delta_1, \delta_{\mu}-\delta_{\mu+1}\,|\, 1\leq i\leq m-1, 1\leq \mu\leq n-1 \}.
$$
For each $\alpha\in\Phi^+$, the corresponding root space is the 1-dimensional vector space spanned by $E_{\alpha}=E_{ab}$,
where $\alpha=\epsilon_a-\epsilon_b$.

Let $(-,-): \mathfrak{h}^\ast_{m|n}\times \mathfrak{h}^\ast_{m|n}\to\mathbb{C}$ denote the symmetric bilinear form defined by
$$
 (\epsilon_i, \epsilon_j)=\delta_{i,j}, \qquad 
 (\epsilon_i, \delta_{\mu})=0,\qquad
 (\delta_{\mu}, \delta_{\nu})=-\delta_{\mu,\nu},
$$
for $i,j\in\{1,\ldots,m\}$ and $\mu,\nu\in\{1,\ldots,n\}$. The graded half-sum of positive roots,
$$
 \rho:=\frac{1}{2}\Big(\sum_{\alpha\in\Phi^+_{\bar0}}\!\alpha-\sum_{\alpha\in\Phi^+_{\bar1}}\!\alpha\Big)
 =\frac{1}{2} \sum_{i=1}^{m} (m-n-2i+1) \epsilon_i + \frac{1}{2} \sum_{\mu=1}^{n} (m+n-2\mu+1) \delta_{\mu},
$$
satisfies
\begin{align}\label{eq: rhoform}
 (\rho, \epsilon_i-\epsilon_j)=j-i,\qquad
 (\rho, \epsilon_i-\delta_{\mu})= m+1-i-\mu,\qquad
 (\rho, \delta_{\mu}-\delta_{\nu})=\mu-\nu.
\end{align}

\subsection{Real form \texorpdfstring{$\mathfrak{u}(p,q|n)$ of $\mathfrak{gl}_{p+q|n}$}{u(p,q|n) of gl(p+q|n)}}
\label{sec: realform}

Let $p,q$ be positive integers such that $m=p+q$. 
On $\mathfrak{gl}_{p+q|n}$, we define the $p,q$-dependent star-operation
\begin{align}\label{eq: star}
 (E_{ab})^{\medstar}:= 
 \begin{cases}
 E_{ba}, & a,b\leq p\ \ \text{or}\ \ a,b>p,\\[.1cm]
 -E_{ba}, & \text{otherwise},
 \end{cases}
\end{align}
and we denote by $\mathfrak{u}(p,q|n)$ the real form of $\mathfrak{gl}_{p+q|n}$ induced by this star-operation.
By construction, it is a real Lie superalgebra, and as its even subalgebra is $\mathfrak{u}(p,q)\oplus \mathfrak{u}(n)$, it is non-compact.
Following the discussion in \secref{Sec: Star}, a $\mathfrak{gl}_{p+q|n}$-module is unitary with respect to \eqnref{eq: star} 
if it is unitary as a module over the corresponding star-superalgebra ${\rm U}(\mathfrak{gl}_{p+q|n})$;
that is, if it carries a positive-definite Hermitian form $\langle-,-\rangle$ satisfying
\begin{align}\label{eq: Hermitian}
 \langle E_{ab}v,\,w\rangle \;=\; \langle v,\,(E_{ab})^{\medstar}w\rangle, \qquad 
 v,w\in V, \quad a,b\in\{1,\ldots,p+q+n\}.
\end{align}

Key to our work, $\mathfrak{u}(p,q|n)$ is exactly the real form of $\mathfrak{gl}_{p+q|n}$ associated with the star-operation
appearing in \eqref{eq: Hermitian}. It follows that
(i) every unitary representation of $\mathfrak{u}(p,q|n)$ extends by complexification to a $\mathfrak{gl}_{p+q|n}$-module 
satisfying \eqref{eq: Hermitian}, 
and conversely, that 
(ii) restricting the action of a unitary $\mathfrak{gl}_{p+q|n}$-module to 
its real part yields a unitary $\mathfrak{u}(p,q|n)$-module. 
Utilizing this, we shall study unitary $\mathfrak{u}(p,q|n)$-modules via the corresponding $\mathfrak{gl}_{p+q|n}$-modules. 

Relative to the real form $\mathfrak{u}(p,q|n)$, 
we define the \textit{non-compact positive root system} of $\mathfrak{gl}_{p+q|n}$ by 
$\Phi^{+}_{\mathrm{nc}}:=\Phi^{+}_{\mathrm{nc},\bar{0}}\cup\Phi^{+}_{\mathrm{nc},\bar{1}}$, where 
$$
 \Phi^{+}_{\mathrm{nc},\bar{0}}:=\{ \epsilon_i-\epsilon_j \,|\, 1\leq i\leq p<j\leq p+q\}, \qquad
 \Phi^{+}_{\mathrm{nc},\bar{1}}:=\{ \epsilon_i-\delta_{\mu} \,|\, 1\leq i\leq p, 1\leq \mu \leq n\}. 
$$
Associated with this, we have the triangular decomposition
\begin{align}
 \mathfrak{gl}_{p+q|n}= \mathfrak{k}_{-}\oplus \mathfrak{k} \oplus \mathfrak{k}_+,
\label{glk}
\end{align}
where 
$$
 \mathfrak{k}_-:={\rm span}\{E_\alpha\,|\,\alpha\in-\Phi^+_{\mathrm{nc}}\},\qquad
 \mathfrak{k}\cong \mathfrak{gl}_{p}\oplus \mathfrak{gl}_{q|n},\qquad
 \mathfrak{k}_+:={\rm span}\{E_\alpha\,|\,\alpha\in\Phi^+_{\mathrm{nc}}\}.
$$
We also define the corresponding \textit{compact positive root system} by 
$\Phi^{+}_{\mathrm{c}}:=\Phi^{+}_{\mathrm{c}, \bar{0}}\cup\Phi^{+}_{\mathrm{c},\bar{1}}$, where 
$$
 \Phi^{+}_{\mathrm{c}, \bar{0}}:= \Phi_{\bar{0}}^+ - \Phi^{+}_{\mathrm{nc},\bar{0}}, \qquad 
 \Phi^{+}_{\mathrm{c},\bar{1}}:= \Phi_{\bar{1}}^+ - \Phi^{+}_{\mathrm{nc},\bar{1}},
$$
and note that $\Phi^{+}_{\mathrm{c}}$ is the positive root system of the $\mathfrak{gl}_{p+q|n}$-subalgebra 
$\mathfrak{k}$.

\subsection{Admissible \texorpdfstring{$\mathfrak{gl}_{p+q|n}$-modules}{gl(p+q|n)-modules}}

We say that a  $\mathfrak{gl}_{p+q|n}$-module is \textit{admissible} if its restriction to $\mathfrak{k}$ in \eqref{glk} decomposes into
a direct sum of finite-dimensional simple $\mathfrak{k}$-modules, each occurring with finite multiplicity. 
Every admissible  module admits a weight-space decomposition with respect to the Cartan subalgebra $\mathfrak{h}$ 
of $\mathfrak{gl}_{p+q|n}$. Throughout, we work with the category of admissible  unitary $\mathfrak{gl}_{p+q|n}$-modules, 
which is the super analogue of the framework used by Enright, Howe and Wallach in their classification of unitary highest-weight 
modules \cite{EHW83}.

The following gives necessary weight conditions for admissible unitary $\mathfrak{gl}_{p+q|n}$-modules (cf. \cite[Lemma 2.1]{FN91}). 
\begin{lemma}\label{lem: wtcon}
Let $V$ be an admissible unitary $\mathfrak{gl}_{p+q|n}$-module, and let 
$\lambda=\sum_{i=1}^{p+q}\lambda_i \epsilon_i + \sum_{\mu=1}^n \omega_{\mu}\delta_{\mu}$ be any weight of $V$. 
Then, $\lambda$ is real, and  
$$
 \lambda_i \leq -\omega_{\mu}\leq \lambda_j
$$
for all $i\in\{1,\ldots,p\}$, $j\in\{p+1,\ldots,p+q\}$, and $\mu\in\{1,\ldots,n\}$. 
\end{lemma}
\begin{proof}
Let $v$ be a nonzero  weight vector with weight $\lambda$. Then $E_{aa}v=\lambda(E_{aa})v$ for all $a$.  
Since $E_{aa}^{\medstar}= E_{aa}$ and the Hermitian form is anti-linear in the first argument, we have 
$$
 \overline{\lambda(E_{aa})}\langle v,v \rangle=\langle E_{aa}v,v\rangle= \langle v,E_{aa}v\rangle =\lambda(E_{aa}) \langle v,v\rangle.
$$
Since $\langle v,v \rangle>0$, we have $\lambda(E_{aa})= \overline{\lambda(E_{aa})}$ for all $a$, so $\lambda$ is a real weight.  

For $i\in\{1,\ldots,p\}$ and $\mu\in\{1,\ldots,n\}$, we have $[E_{i, m+\mu}, E_{m+\mu, i}]= E_{ii}+ E_{m+\mu, m+\mu}$ 
and $(E_{i, m+\mu})^{\medstar}= - E_{m+\mu, i}$, so
\begin{align*}
 (\lambda_i+\omega_{\mu})\langle v, v\rangle
 &= \langle [E_{i, m+\mu}, E_{m+\mu, i}] v, v\rangle
 = \langle E_{i, m+\mu} E_{m+\mu, i} v,  v\rangle + \langle E_{m+\mu, i} E_{i, m+\mu} v,  v\rangle 
 \\[.1cm]
 &= -\langle E_{m+\mu, i} v,  E_{m+\mu, i} v\rangle - \langle  E_{i, m+\mu} v,  E_{i, m+\mu} v\rangle
 \leq 0,
\end{align*}
hence $\lambda_i\leq - \omega_{\mu}$. Similarly, using $[E_{j, m+\mu}, E_{m+\mu, j}]= E_{jj}+ E_{m+\mu, m+\mu}$ 
and $(E_{j, m+\mu})^{\medstar}=E_{m+\mu, j}$, it follows that $-\omega_{\mu}\leq \lambda_j$ for all $j\in\{p+1,\ldots,m\}$ 
and all $\mu\in\{1,\ldots,n\}$. 
\end{proof}
\begin{remark}
Corresponding to $n=0$, \lemref{lem: wtcon} does not apply to the non-compact real form $\mathfrak{u}(p,q)$,
as the inequalities $\lambda_i\leq \lambda_j$ do not generally hold for $i\in\{1,\ldots,p\}$ and $j\in\{p+1,\ldots,p+q\}$.
\end{remark}
An important consequence of \lemref{lem: wtcon} is the following result (cf. \cite[Proposition 2.2]{FN91}). 
\begin{proposition}\label{prop: highest}
Let $V$ be an admissible unitary simple $\mathfrak{gl}_{p+q|n}$-module. Then, $V$ is a highest-weight module with highest weight 
$\Lambda=\sum_{i=1}^{p+q}\lambda_i \epsilon_i + \sum_{\mu=1}^n \omega_{\mu}\delta_{\mu}$ satisfying 
$$
 \lambda_{p+1}\geq \cdots \geq \lambda_m \geq -\omega_n \geq \cdots \geq -\omega_1\geq \lambda_1 \geq \cdots \geq \lambda_p.
$$
\end{proposition}
\begin{proof}
As $V$ is simple and admits a weight-space decomposition, there exists a nonzero weight vector $v\in V$ with weight $\lambda$ 
such that $V= \mathrm{U}(\mathfrak{gl}_{m|n})v$. By \lemref{lem: wtcon}, we have $\lambda_i-\lambda_j\leq 0$
for all $i\in\{1,\ldots,p\}$ and $j\in\{p+1,\ldots,p+q\}$, and as the action of $E_{ij}$ on $v$ increases the weight 
$\lambda_i-\lambda_j$ by 2, there exists a non-negative integer $s_{ij}$ such that $E_{ij}^{s_{ij}}v\neq 0$ and $E_{ij}^{s_{ij}+1}v=0$.
It follows that there exists a $\Phi^{+}_{\mathrm{nc},\bar{0}}$-highest-weight vector $w$; 
that is, $E_{\alpha}w=0$ for all $\alpha\in \Phi^{+}_{\mathrm{nc},\bar{0}}$. 
Since $E^2_{\beta}=0$ and $[E_{\alpha}, E_{\beta}]=0$ 
for any $\alpha\in \Phi^{+}_{\mathrm{nc},\bar{0}}$ and $\beta\in \Phi^{+}_{\mathrm{nc},\bar{1}}$, 
there exists a $\Phi_{\mathrm{nc}}^+$-highest-weight vector $w'$ such that $E_{\alpha}w'= 0$ 
for all $\alpha\in \Phi_{\mathrm{nc}}^+$. 
As $V$ is admissible unitary, $w'$ generates a finite-dimensional 
$\mathfrak{k}$-module, so there exists $X\in \mathrm{U}(\mathfrak{k})$ such that $u=Xw'$ is a $\Phi_{c}^+$-highest-weight vector.
As $[\mathfrak{k}_+, \mathfrak{k}]\subseteq \mathfrak{k}_+$, $u$ is also a $\Phi_{\mathrm{nc}}^+$-highest-weight vector, 
hence a $\Phi^+$-highest-weight vector, so $V$ is a highest-weight module.
The inequalities satisfied by the weight components follow from \lemref{lem: wtcon} and  the admissibility of $V$.
\end{proof}

\subsection{Unitary highest-weight modules}
\label{Sec:Uhwm}

Let $\Dpqn$ denote the set of 
\textit{$\Phi^+_{\mathrm{c}}$-dominant integral weights}; that is, weights of the form
\begin{align}\label{def: Lambda}
 (\lambda_1, \dots , \lambda_{p+q}, \omega_1, \dots, \omega_n)
 = \sum_{i=1}^{p+q}\lambda_i\epsilon_i + \sum_{\mu=1}^n \omega_{\mu} \delta_{\mu},
\end{align}
where $\lambda_i, \omega_{\mu}\in \mathbb{R}$ for all $i,\mu$, and such that
\begin{align}\label{eq: k-high}
\begin{aligned}
 \lambda_{i}-\lambda_{i+1}&\in \mathbb{Z}_+, &&\quad i \in \{1,\dots, p-1\}\cup\{p+1,\dots, p+q-1\}, \\[.1cm]
 \omega_{\mu}-\omega_{\mu+1}&\in \mathbb{Z}_+, &&\quad \mu\in \{ 1,\dots, n-1\}. 
\end{aligned}
\end{align}

The unique simple $\mathfrak{k}$-module of highest weight $\Lambda\in \Dpqn$ is denoted by $L_0(\Lambda)$. 
We can turn $L_0(\Lambda)$ into a $(\mathfrak{k}\oplus\mathfrak{k}_{+})$-module by letting $\mathfrak{k}_+$ act by zero,
and we use this to define the highest-weight $\mathfrak{gl}_{p+q|n}$-module
\begin{align}\label{eq: hwmod}
 V(\Lambda):= {\rm U}(\mathfrak{gl}_{p+q|n}) \otimes_{\,{\rm U}(\mathfrak{k}\,\oplus\,\mathfrak{k}_{+})} \!L_0(\Lambda). 
\end{align}
The quotient $L(\Lambda)$ of $V(\Lambda)$ by its unique maximal proper submodule thus yields an irreducible
$\mathfrak{gl}_{p+q|n}$-module with highest weight $\Lambda$. 
By the PBW theorem for $\mathfrak{gl}_{p+q|n}$, we have $V(\Lambda)={\rm U}(\mathfrak{k}_{-}) \otimes L_0(\Lambda)$. 
Together with the convention ${\rm deg} (E_{ai})=1$ for $a\in\{p+1,\ldots,p+q+n\}$ and $i\in\{1,\ldots,p\}$,
this yields the $\mathbb{Z}_+$-grading
\begin{align}\label{eq: VLZ}
 V(\Lambda)= \bigoplus_{k\in \mathbb{Z}_+} V_k(\Lambda),
\end{align}
where $V_0(\Lambda)=L_{0}(\Lambda)$ and each $V_{k}(\Lambda)$ is a finite-dimensional $\mathfrak{k}$-module. 
\begin{proposition}\label{prop: ineqcon}
Let $\Lambda=(\lambda_1, \dots , \lambda_{p+q}, \omega_1, \dots, \omega_n)\in \Dpqn$, and suppose $L(\Lambda)$ 
is unitary. Then, the following inequalities hold: 
\begin{enumerate}
\item[{\rm (1)}] $(\Lambda, \alpha)\geq 0$ for all $\alpha\in \Phi_{\mathrm{c}}^+$.
\vspace{0.15cm}
\item[{\rm (2)}] $(\Lambda, \beta)\leq 0$ for all $\beta\in \Phi_{\mathrm{nc}}^+$. 
\vspace{0.15cm}
\item[{\rm (3)}] $\lambda_{p+1}\geq\cdots\geq\lambda_{p+q}\geq-\omega_n\geq\cdots
 \geq-\omega_1\geq\lambda_1\geq\cdots\geq\lambda_p$. 
\end{enumerate} 
\end{proposition}
\begin{proof}
Let $v_{\Lambda}$ be a highest-weight vector of $L(\Lambda)$.
If $\alpha\in \Phi_{\mathrm{c}, \bar{0}}^+$, then (i) $\alpha=\epsilon_i-\epsilon_j$ for some
$i,j\in\{1,\ldots,p\}$ or $i,j\in\{p+1,\ldots,p+q\}$ such that $i<j$, 
in which case $(\Lambda,\alpha)=\lambda_i-\lambda_j\geq0$ by \eqnref{eq: k-high}, 
or (ii) $\alpha=\delta_\mu-\delta_\nu$ for some $\mu,\nu\in\{1,\ldots,n\}$ such that $\mu<\nu$, in which case 
$(\Lambda,\alpha)=\omega_\mu-\omega_\nu\geq0$, again by \eqnref{eq: k-high}.
If $\alpha\in \Phi_{\mathrm{c},\bar{1}}^+$, then $\alpha= \epsilon_i-\delta_{\mu}$ for some
$i\in\{p+1,\ldots,p+q\}$ and $\mu\in\{1,\ldots,n\}$, so $(E_{\mu i})^{\medstar}= E_{i\mu}$ and 
$$
 0\leq \langle E_{\mu i}v_{\Lambda}, E_{\mu i}v_{\Lambda}\rangle 
 =\langle v_{\Lambda}, E_{i\mu}E_{\mu i}v_{\Lambda}\rangle 
 = \langle v_{\Lambda}, (E_{ii}+E_{\mu\mu})v_{\Lambda}\rangle 
 = (\lambda_i+\omega_\mu)\,\langle v_{\Lambda}, v_{\Lambda} \rangle,
$$
hence $0\leq \lambda_i+\omega_\mu=(\Lambda, \epsilon_i-\delta_{\mu})=(\Lambda,\alpha)$. 
If $\beta\in \Phi_{\mathrm{nc},\bar{0}}^+$, then $\beta= \epsilon_i-\epsilon_j$ for some
$i\in\{1,\ldots,p\}$ and $j\in\{p+1,\ldots,p+q\}$, so $(E_{ij})^{\medstar}=-E_{ji}$ and
$0\leq \langle E_{ji}v_{\Lambda}, E_{ji}v_{\Lambda}\rangle 
 =-(\lambda_i-\lambda_j)\,\langle v_{\Lambda}, v_{\Lambda} \rangle$,
hence $0\geq\lambda_i-\lambda_j=(\Lambda,\epsilon_i-\epsilon_j)=(\Lambda,\beta)$.
If $\beta\in \Phi_{\mathrm{nc},\bar{1}}^+$, then $\beta= \epsilon_i-\delta_\mu$ for some
$i\in\{1,\ldots,p\}$ and $\mu\in\{1,\ldots,n\}$, so $(E_{i\mu})^{\medstar}=-E_{\mu i}$ and
$0\leq \langle E_{\mu i}v_{\Lambda}, E_{\mu i}v_{\Lambda}\rangle 
 =-(\lambda_i+\omega_\mu)\,\langle v_{\Lambda}, v_{\Lambda} \rangle$,
hence $0\geq\lambda_i+\omega_\mu=(\Lambda,\epsilon_i-\delta_\mu)=(\Lambda,\beta)$.
Part (3) follows by combining \eqnref{eq: k-high} with $0\leq \lambda_i+\omega_\mu$ for $i\in\{p+1,\ldots,p+q\}$
and $0\geq\lambda_i+\omega_\mu$ for $i\in\{1,\ldots,p\}$.
\end{proof}

By \propref{prop: ineqcon}, $L(\Lambda)$ is an infinite-dimensional unitary $\mathfrak{gl}_{p+q|n}$-module unless
$$
 \lambda_{p+1} = \cdots = \lambda_{p+q} = -\omega_n = \cdots = -\omega_1 = \lambda_1 = \cdots = \lambda_p.
$$
(If these equalities all hold, then $L(V)$ is a 1-dimensional unitary module.)
Moreover, it inherits the $\mathbb{Z}_+$-grading \eqref{eq: VLZ}:
\begin{align}\label{eqn: decomp}
 L(\Lambda)=\bigoplus_{k\in \mathbb{Z}_+}L_k(\Lambda),
\end{align}
where each summand $L_k(\Lambda)$ is a finite-dimensional $\mathfrak{k}$-module.

For each nonzero $s\in \mathbb{R}$, we also introduce the $1$-dimensional $\mathfrak{gl}_{p+q|n}$-module
$\mathbb{C}_{s}$ with action given by 
$$
 E_{ab}\cdot 1=\delta_{a,b}(-1)^{[a]}s\cdot 1
$$
for all $a,b\in\{1,\ldots,p+q+n\}$.
Accordingly, $\mathbb{C}_{s}$ has the unique weight
$$
\Lambda_s=(\underbrace{s,\ldots,s}_{p+q},\underbrace{-s,\ldots,-s}_{n})
$$
and is clearly a unitary module. For any $\Lambda\in \Dpqn$, the shifted weight
$$
 \Lambda^{(s)}:=\Lambda+\Lambda_s
 = (\lambda_1 + s, \dots, \lambda_{p+q} + s, \omega_1 - s, \dots, \omega_{n}-s)
$$
is the highest weight of the $\mathfrak{gl}_{p+q|n}$-module $V(\Lambda)\otimes\mathbb{C}_s$.

\subsection{Finite-dimensional unitary modules}
\label{subsec: findim}

For later use, here we recall the classification of finite-dimensional unitary $\mathfrak{gl}_{m|n}$-modules. 

First, we have the triangular decomposition 
\begin{align}
 \mathfrak{gl}_{m|n}=\mathfrak{g}_{-1}\oplus \mathfrak{g}_{0}\oplus \mathfrak{g}_{1},
\label{triangular}
\end{align}
where $\mathfrak{g}_{0}:=\mathfrak{gl}_{m} \oplus \mathfrak{gl}_{n}$ is the even subalgebra,
while $\mathfrak{g}_{\pm1}:=\mathrm{span}\{E_\alpha\,|\,\alpha\!\in\!\pm\,\Phi_{\bar 1}^+\}$.
Second, every finite-dimensional simple $\mathfrak{gl}_{m|n}$-module is uniquely characterised 
by its highest weight $\Lambda=(\lambda_1,\dots,\lambda_m, \omega_1, \dots, \omega_n)$, which must satisfy
$\lambda_i-\lambda_{i+1}\in \mathbb{Z}_+$ and $\omega_\mu-\omega_{\mu+1}\in \mathbb{Z}_+$ 
for all $i\in\{1,\ldots,m-1\}$ and all $\mu\in\{1,\ldots,n-1\}$. That is, $\Lambda$ is a dominant integral weight of
$\mathfrak{gl}_{m} \oplus \mathfrak{gl}_{n}$, and we denote the corresponding simple module by $L(\Lambda)$.

Following \cite{Kac77a}, $L(\Lambda)$ and the corresponding weight 
$\Lambda$ are said to be \textit{typical} if 
$$
 \prod_{\alpha\in \Phi^+_{\bar{1}}}(\Lambda+\rho, \alpha)\neq 0;
$$ 
they are called \textit{atypical} otherwise. 
Associated to the same weight $\Lambda$, we also have the \textit{Kac module} defined by
\begin{align}
 K(\Lambda):= {\rm U}(\mathfrak{gl}_{m|n}) \otimes_{\,{\rm U}(\mathfrak{g}_{0}\,\oplus\,\mathfrak{g}_{1}\!)} \!L_{0}(\Lambda),
\label{Kac}
\end{align}
where $L_0(\Lambda)$ is a simple $\mathfrak{g}_{0}$-module equipped with trivial $\mathfrak{g}_1$-action.
If $\Lambda$ is typical, then $L(\Lambda)$ is isomorphic to $K(\Lambda)$, while if $\Lambda$ is atypical, 
then $K(\Lambda)$ is non-simple and $L(\Lambda)$ is isomorphic to $K(\Lambda)/M(\Lambda)$ where $M(\Lambda)$ is the unique 
maximal proper submodule of $K(\Lambda)$.

Following \cite{SNR77,GZ90b}, there exists an induced non-degenerate Hermitian form $\langle -,- \rangle$ on $L(\Lambda)$ 
(unique up to a scalar multiple) which is positive-definite on $L_0(\Lambda)$ and such that for all $v,w\in L(\Lambda)$ and 
all $a,b\in\{1,\ldots,m+n\}$,
$$
 \langle E_{ab}v,w \rangle= (-1)^{\theta ([a]+[b])} \langle v, E_{ba}w \rangle,
$$
for some fixed $\theta\in \{0,1\}$.
With respect to this form, the $\mathbb{Z}$-graded decomposition \eqref{eqn: decomp} is orthogonal.
Following \cite{GZ90b}, we say that $L(\Lambda)$ is a \textit{type-1} (respectively \textit{type-2}) \textit{unitary module}
if the Hermitian form is positive-definite on $L(\Lambda)$ for $\theta=0$ (respectively $\theta=1$).

Let $D^+_{m|n}$ denote the set of real dominant integral weights of $\mathfrak{gl}_m\oplus\mathfrak{gl}_n$. 
\begin{theorem}[\cite{GZ90b}]\label{thm: type1unitary}
Let $\Lambda\in D^+_{m|n}$. Then, $L(\Lambda)$ is type-1 unitary if and only if
one of the following conditions holds:
\begin{enumerate}
\item[{\rm (1)}] $(\Lambda+ \rho, \epsilon_m-\delta_n)>0$.
\vspace{0.15cm}
\item[{\rm (2)}] There exists $\mu\in \{1,\dots, n\}$ such that 
$(\Lambda + \rho, \epsilon_m-\delta_{\mu}) = (\Lambda, \delta_{\mu}-\delta_n)=0$.
\end{enumerate}
\end{theorem}
\begin{remark}\label{rmk: typ}
If $(\Lambda+\rho, \epsilon_m-\delta_{n})>0$, then for every $i\in\{1,\ldots,m\}$ and $\mu\in\{1,\ldots,n\}$, 
\begin{align*}
 (\Lambda+ \rho,\epsilon_i-\delta_{\mu}) 
 &= (\Lambda+\rho,\epsilon_i-\epsilon_m + \epsilon_m-\delta_n+ \delta_n-\delta_{\mu})
 \\[.1cm]
 &=(\lambda_i-\lambda_m + m-i)+(\Lambda+\rho, \epsilon_m-\delta_{n})+ (\omega_{\mu}-\omega_n + n-\mu).
\end{align*}
As this is strictly positive, the conditions (1) and (2) in \thmref{thm: type1unitary} are mutually exclusive,
so any simple unitary module $L(\Lambda)$ with $\Lambda\in D^+_{m|n}$ is either typical or atypical, 
depending on whether condition (1) or (2) is satisfied.
\end{remark}

Type-1 and type-2 unitary modules are related by duality, although this relationship is not immediately evident from their definitions.
However, it was shown in \cite{GZ90b} (see also \propref{prop: dualstar}) that a simple $\mathfrak{gl}_{m|n}$-module $L(\Lambda)$ 
is type-1 unitary if and only if the dual module $L(\Lambda)^{\ast}$ is type-2 unitary. 

Note that $L(\Lambda)$ has a natural $\mathbb{Z}$-grading $L(\Lambda)= \bigoplus_{k=0}^{d_{\Lambda}} L_k(\Lambda)$ inherited 
from the Kac module, where $0\leq d_{\Lambda}\leq mn$ and each $L_{k}(\Lambda)$ is a $\mathfrak{g}_0$-module. 
In particular, $L_{0}(\Lambda)$ is simple with highest weight $\Lambda$.
Following \cite{GZ90b}, the $\mathfrak{g}_0$-highest weight of the minimal $\mathbb{Z}$-graded component 
$L_{d_{\Lambda}}(\Lambda)$ is determined as follows. 
If $\Lambda$ is typical, we set $\mu=n+1$; otherwise we set $\mu$ to be the odd index 
satisfying condition (2) in \thmref{thm: type1unitary}, and introduce
$$
 \mu_{m}=\mu-1,\qquad
 \mu_i = {\rm min}\{n, \mu_m+(\Lambda,\epsilon_i-\epsilon_m)\}, \qquad 
 i=1,\ldots,m-1.
$$
Then, the $\mathfrak{g}_0$-highest weight of $L_{d_{\Lambda}}(\Lambda)$ is 
$$
 \bar{\Lambda}=\Lambda- \sum_{i=1}^m \sum_{\nu=1}^{\mu_i} (\epsilon_i-\delta_{\nu}),
$$
with corresponding $\mathfrak{g}_0$-highest-weight vector given by
\begin{align}\label{eq: Omega}
 \Omega_{m|n}=\Big(\prod_{i=1}^m \prod_{\nu=1}^{\mu_i}E_{m+\nu, i}\Big)v_{\Lambda}. 
\end{align}
A classification of the type-2 unitary modules now follows. 
\begin{theorem}[\cite{GZ90b}]\label{thm: type2}
Let $\Lambda\in D^+_{m|n}$. Then, $L(\Lambda)$ is type-2 unitary if and only if
one of the following conditions holds:
\begin{enumerate}
\item[{\rm (1)}] $(\Lambda+ \rho, \epsilon_1-\delta_1)<0$.
\vspace{0.15cm}
\item[{\rm (2)}] There exists $k\in \{1,\dots, m\}$ such that 
$(\Lambda + \rho, \epsilon_k-\delta_{1}) = (\Lambda, \epsilon_1-\epsilon_k)=0.$ 
\end{enumerate}
\end{theorem}

\begin{remark}
As in \thmref{thm: type1unitary}, the two conditions in \thmref{thm: type2} are mutually exclusive.
\end{remark}

We conclude this section with the following well-known facts (cf. \lemref{lem: wtcon}). 
Let $L(\Lambda)$ be a finite-dimensional simple $\mathfrak{gl}_m$-module. 
Then, the module $L(\Lambda)$ is unitary if and only if $\Lambda$ is a real dominant integral weight. 
A Hermitian form on $L(\Lambda)$ is positive-definite and contravariant if it satisfies the following two conditions:
\begin{enumerate}
\vspace{0.15cm}
\item[(i)] $\langle v_{\Lambda}, v_{\Lambda}\rangle>0$, where $v_{\Lambda}$ is a highest-weight vector of $L(\Lambda)$;
\vspace{0.2cm}
\item[(ii)]
$\langle E_{ab}v,w \rangle= \langle v, E_{ba} w\rangle$ for all $v,w\in L(\Lambda)$ and all $a,b\in \{1, \dots, m\}$. 
\vspace{0.15cm}
\end{enumerate}
Such a form is unique up to scaling.

\section{Main classification result}
\label{sec: class}

Here, we present our main result: a classification of all unitary simple highest-weight $\mathfrak{u}(p,q|n)$-modules of
the form $L(\Lambda)$ with $\Lambda\in \Dpqn$; see \thmref{thm: main} below.
Here and in the remainder of this paper, we use the notation $m=p+q$.
\begin{lemma}\label{lem: U}
Let $\Lambda\in \Dpqn$. Then, the following six conditions are mutually exclusive:
\begin{enumerate}
\vspace{0.1cm}
\item[\em \textbf{(U1)}] $(\Lambda+\rho, \epsilon_m-\delta_n)>0$ and $(\Lambda+\rho, \epsilon_1-\delta_1)<0$.
\vspace{0.3cm}
\item[\em\textbf{(U2)}] $(\Lambda+\rho, \epsilon_m-\delta_n)>0$, and there exists $i\in \{1,\dots,p\}$ such that
$(\Lambda+\rho, \epsilon_i-\delta_1)=(\Lambda, \epsilon_i-\epsilon_1)=0$.
\vspace{0.3cm}
\item[\em \textbf{(U3)}] $(\Lambda+\rho, \epsilon_1-\delta_1)<0$, and there exists $\mu\in \{2,\dots, n\}$ such that 
$(\Lambda+\rho, \epsilon_m-\delta_\mu )=(\Lambda, \delta_\mu-\delta_n)=0$.
\vspace{0.3cm}
\item[\em \textbf{(U4)}] There exists $\mu\in \{2,\dots, n\}$ such that 
$(\Lambda+\rho, \epsilon_m-\delta_\mu )=(\Lambda, \delta_\mu-\delta_n)=0$,
and there exists $i\in \{1,\dots,p\}$ such that
$(\Lambda+\rho, \epsilon_i-\delta_1)=(\Lambda, \epsilon_i-\epsilon_1)=0$.
\vspace{0.3cm}
\item[\em \textbf{(U5)}] $(\Lambda+\rho, \epsilon_m-\delta_1)=(\Lambda, \delta_1-\delta_n)= 0$, 
and there exists $j\in\{p,\dots, m-1\}$ such that
$(\Lambda, \epsilon_1-\delta_1)<1-j$ and $(\Lambda, \epsilon_{j+1}-\epsilon_m)=0$.
\vspace{0.3cm}
\item[\em \textbf{(U6)}] $(\Lambda+\rho, \epsilon_m-\delta_1)=(\Lambda, \delta_1-\delta_n)= 0$,
and there exist $i\in \{1,\dots,p\}$ and $j\in\{p,\dots, m-1\}$ such that 
$(\Lambda, \epsilon_i-\epsilon_1)=(\Lambda, \epsilon_{j+1}-\epsilon_m)=0$
and $(\Lambda, \epsilon_i-\delta_1)=i-j$.
\vspace{0.1cm}
\end{enumerate}
\end{lemma}
\begin{proof}
Throughout, we use the identities \eqref{eq: rhoform} and the dominance condition \eqref{eq: k-high}.
If $(\Lambda+ \rho, \epsilon_m-\delta_n)>0$, then for any $\mu\in \{1, \dots, n\}$,
\begin{align*}
(\Lambda+ \rho, \epsilon_m-\delta_\mu)&= (\Lambda+ \rho, \epsilon_m-\delta_n)+ (\Lambda+ \rho, \delta_n-\delta_\mu)
\\[.1cm] 
&= (\Lambda+ \rho, \epsilon_m-\delta_n)+ (\Lambda, \delta_n-\delta_{\mu})+ n- \mu>0.
\end{align*}
Consequently, \textbf{(U1)} and  \textbf{(U2)} are each disjoint from \textbf{(U3)}--\textbf{(U6)}. 

Similarly, if $(\Lambda+ \rho, \epsilon_1-\delta_1)<0$, then for any $i\in \{1, \dots, p\}$, 
\begin{align*}
(\Lambda+ \rho, \epsilon_i-\delta_1)&=(\Lambda+ \rho, \epsilon_i-\epsilon_1) + (\Lambda+ \rho, \epsilon_1-\delta_1)
\\[.1cm] 
&= (\Lambda, \epsilon_i-\epsilon_1) + (1-i) + (\Lambda+ \rho, \epsilon_1-\delta_1)<0.
\end{align*}
Therefore, \textbf{(U1)} and \textbf{(U2)} are disjoint, and \textbf{(U3)} and \textbf{(U4)} are disjoint. 

If $(\Lambda+\rho, \epsilon_m-\delta_1)=(\Lambda, \delta_1-\delta_n)= 0$, then, for any $\mu\in \{2, \dots, n\}$,
\begin{align*}
 (\Lambda+ \rho, \epsilon_m-\delta_\mu)&= (\Lambda+ \rho, \epsilon_m-\delta_1) + (\Lambda+ \rho, \delta_1-\delta_\mu)
 = (\Lambda, \delta_1-\delta_{\mu})+  1 - \mu
 \\[.1cm] 
 &=1-\mu < 0.
\end{align*}
Hence, \textbf{(U5)} and \textbf{(U6)} are each disjoint from \textbf{(U3)}--\textbf{(U4)}. 

Finally, suppose that \textbf{(U6)} is satisfied, so there exist $i\in \{1,\dots,p\}$ and $j\in\{p,\dots, m-1\}$ such that 
$(\Lambda, \epsilon_i-\epsilon_1)=(\Lambda, \epsilon_{j+1}-\epsilon_m)=0$
and $(\Lambda, \epsilon_i-\delta_1)=i-j$. Then,
$$
 (\Lambda, \epsilon_1-\delta_1)= (\Lambda, \epsilon_1-\epsilon_i)+ (\Lambda, \epsilon_i- \delta_1)= i-j \geq 1-j, 
$$
which contradicts the condition in \textbf{(U5)}. It follows that \textbf{(U5)} and \textbf{(U6)} are disjoint. 

Combining the above, the conditions \textbf{(U1)}--\textbf{(U6)} are seen to be mutually exclusive.
\end{proof}
\begin{theorem}\label{thm: main}
Let $\Lambda\in \Dpqn$. Then, the $\mathfrak{u}(p,q|n)$-module $L(\Lambda)$ is unitary if and only if 
$\Lambda$ satisfies one of the conditions \emph{\textbf{(U1)}--\textbf{(U6)}} in \lemref{lem: U}.
\end{theorem}
Sections \ref{sec: nec}--\ref{sec: suff} are devoted to the proof of Theorem \ref{thm: main}. In the course of the proof, for each of 
the conditions \textbf{(U1)}--\textbf{(U6)}, we will establish the existence of a highest weight $\Lambda$ satisfying the condition.

Theorem \ref{thm: main} also recovers the classical criterion from \cite{EHW83} for the unitarity of the 
$\mathfrak{u}(p,q)$-module $L(\Lambda)$ with $\Lambda\in D_{p,q}^+$, 
where $D_{p,q}^+=D_{p,q|0}^+$ denotes the set of real dominant integral weights of $\mathfrak{gl}_p\oplus\mathfrak{gl}_q$.
Specifically, this module is unitary if and only if there exist $i\in\{1,\ldots,p\}$ and $j\in\{1,\ldots,q\}$ such that 
$$
 (\Lambda, \epsilon_1-\epsilon_i)=(\Lambda, \epsilon_{m-j+1}-\epsilon_{m})=0,
$$
and such that one of the following two conditions holds: 
\begin{enumerate}
\vspace{0.1cm}
\item[\textbf{(C1)}] $\lambda_{m}-\lambda_1= m-j-i$.
\vspace{0.2cm}
\item[\textbf{(C2)}] $\lambda_{m}-\lambda_1\in\mathbb{R}$ and $\lambda_{m}-\lambda_1>\mathrm{min}\{m-i, m-j\}-1$.
\vspace{0.1cm}
\end{enumerate}
By eliminating the odd coordinates, condition \textbf{(C1)} is obtained from \textbf{(U6)}, 
while \textbf{(C2)} follows from a combination of \textbf{(U2)} and \textbf{(U5)}.

\section{Necessity}
\label{sec: nec}

In this section, we show that one of the conditions \textbf{(U1)}--\textbf{(U6)} \textit{necessarily} holds 
if the $\mathfrak{u}(p,q|n)$-module $L(\Lambda)$ is unitary. 
Our arguments are based on the unitarity conditions for the finite-dimensional simple modules over the 
Lie superalgebras $\mathfrak{gl}_{p|n}$ and $\mathfrak{gl}_{q|n}$.
We recall the notation $m=p+q$.

Clearly, $\mathfrak{gl}_{p|n}$ embeds into $\mathfrak{gl}_{m|n}$, having simple root system 
$$
 \Delta_{p|n}= \{\epsilon_i-\epsilon_{i+1}, \epsilon_p-\delta_1, \delta_{\mu}-\delta_{\mu+1}\,|\, 
 i=1,\ldots,p-1;\, \mu=1,\ldots,n-1 \}
$$
and graded half-sum of positive roots
$$
 \rho_{p|n}= \frac{1}{2} \sum_{i=1}^p (p-n-2i+1)\epsilon_i+ \frac{1}{2}\sum_{\mu=1}^n (p+n-2\mu+1)\delta_{\mu}.
$$
Similarly, $\mathfrak{gl}_{q|n}$ has simple root system 
$$
 \Delta_{q|n}= \{\epsilon_i-\epsilon_{i+1}, \epsilon_m-\delta_1, \delta_{\mu}-\delta_{\mu+1}\,|\, 
 i=p+1,\ldots,m-1;\, \mu=1,\ldots,n-1\}
$$
and graded half-sum of positive roots
$$
 \rho_{q|n}= \frac{1}{2}\sum_{i=p+1}^m (q-n-2(i-p)+1) \epsilon_i+\frac{1}{2}\sum_{\mu=1}^{n}(q+n-2\mu+1)\delta_{\mu}.
$$
\begin{lemma}\label{lem: kmodL}
Let $\Lambda\in \Dpqn$. Then, the $\mathfrak{k}$-module $L_0(\Lambda)$ is type-1 unitary if and only if
one of the following conditions holds:
\begin{enumerate}
\item[{\rm (1)}] $(\Lambda+ \rho, \epsilon_m-\delta_n)>0$.
\vspace{0.15cm}
\item[{\rm (2)}] There exists $\mu\in \{1,\dots, n\}$ such that $(\Lambda+\rho, \epsilon_m-\delta_\mu )=(\Lambda, \delta_\mu-\delta_n)=0$.
\end{enumerate}
\end{lemma}
\begin{proof}
$\Lambda$ decomposes into $\mathfrak{gl}_p$- and $\mathfrak{gl}_{q|n}$-weights, as $\Lambda=(\Lambda', \Lambda'')$.
If $L_{0}(\Lambda)$ is a simple $\mathfrak{k}$-module, then $L_0(\Lambda)\cong L^{p}(\Lambda') \otimes L^{q|n}(\Lambda'')$, 
where $L^{p}(\Lambda')$ and $L^{q|n}(\Lambda'')$ are $\mathfrak{gl}_p$- and $\mathfrak{gl}_{q|n}$-modules with highest weights 
$\Lambda'$ and $\Lambda''$. As per the remarks at the end of 
Section \ref{subsec: findim}, $L^p(\Lambda')$ is unitary if and only if $\Lambda'$ is real and $\mathfrak{gl}_p$-dominant integral. 
\thmref{thm: type1unitary} implies that $L(\Lambda'')$ is type-1 unitary if and only if either
$(\Lambda''+ \rho_{q|n}, \epsilon_m-\delta_n)>0$ or there exists $\mu\in \{1,\dots, n\}$ 
such that $(\Lambda''+\rho_{q|n}, \epsilon_m-\delta_\mu)=(\Lambda'', \delta_\mu-\delta_n)=0$.
As
$$
 (\Lambda+\rho, \epsilon_m-\delta_{\mu})= (\Lambda''+\rho_{q|n}, \epsilon_m-\delta_\mu), \qquad 
 \mu=1,\ldots,n,
$$ 
the result follows. 
\end{proof}

Let $v_{\Lambda}$ be a highest-weight vector of $L(\Lambda)$. Following the construction in \eqref{eq: Omega}, 
in accordance with the embedding 
$\mathfrak{gl}_{q|n}\subset \mathfrak{gl}_{p+q|n}$, we define the $(\mathfrak{gl}_q\oplus \mathfrak{gl}_n)$-highest-weight vector
$$
 \Omega_{q|n}=\Big(\prod_{j=p+1}^{m} \prod_{\nu=1}^{\mu_j}E_{m+\nu,j}\Big)v_{\Lambda}.
$$
Clearly, $\Omega_{q|n}$ is a $\mathfrak{gl}_p$-highest-weight vector, and as
$E_{i,m+\mu}\Omega_{q|n}=0$ for all $i\in\{1,\ldots,p\}$ and all $\mu\in\{1,\ldots,n\}$,
$\Omega_{q|n}$ is also a $(\mathfrak{gl}_{p|n}\oplus \mathfrak{gl}_q)$-highest-weight vector.
If $L(\Lambda)$ is unitary, then every $\mathfrak{gl}_{p|n}$-submodule of $L(\Lambda)$ is type-2 unitary.
\begin{proposition}\label{prop: nec}
Let $\Lambda\in D^{+}_{p+q|n}$, and suppose $L(\Lambda)$ is unitary.
Then, $\Lambda$ satisfies one of the conditions \emph{\textbf{(U1)}--\textbf{(U6)}}.
\end{proposition}
\begin{proof}
If $L(\Lambda)$ is unitary, then the $\mathfrak{k}$-submodule $L_0(\Lambda)$ is type-1 unitary. Consequently, 
the weight $\Lambda$ satisfies one of the conditions in \lemref{lem: kmodL}. If condition (1) in that lemma holds, or condition (2) 
holds with $\mu>1$, then $\Omega_{q|n}$ is a $(\mathfrak{gl}_{p|n}\oplus \mathfrak{gl}_q)$-highest-weight vector of weight
$$
 \Lambda_{q|n}= \Lambda- \sum_{i>p}^{m}\sum_{\nu=1}^{\mu_i} (\epsilon_i-\delta_{\nu}),
$$ 
where $\mu_i\geq \mu$. The highest weight $\Lambda_{q|n}$ remains dominant integral for 
$\mathfrak{gl}_p\oplus \mathfrak{gl}_q\oplus\mathfrak{gl}_n$. 
Since $L(\Lambda)$ is unitary and hence semisimple as a $(\mathfrak{gl}_{p|n}\oplus \mathfrak{gl}_q)$-module, it contains 
a finite-dimensional type-2 unitary simple $(\mathfrak{gl}_{p|n}\oplus \mathfrak{gl}_q)$-submodule generated by $\Omega_{q|n}$. 
By \thmref{thm: type2}, we have either
$$
 (\Lambda_{q|n}+\rho_{p|n}, \epsilon_1-\delta_1)
 = (\Lambda, \epsilon_1-\delta_1)+ m-1
 = (\Lambda+\rho, \epsilon_1-\delta_1)<0,
$$
in which case $\Lambda$ satisfies \textbf{(U1)} or \textbf{(U3)}, or there exists $i\in\{1,\ldots,p\}$ such that 
\begin{align*}
 (\Lambda_{q|n}+\rho_{p|n}, \epsilon_i-\delta_1)
 = (\Lambda, \epsilon_i-\delta_i) + m-i 
 = (\Lambda+\rho, \epsilon_i-\delta_1)
 &=0,
 \\[.1cm]
 (\Lambda_{q|n}, \epsilon_1-\epsilon_i)
 = (\Lambda, \epsilon_1-\epsilon_i)
 &=0,
\end{align*}
in which case $\Lambda$ satisfies \textbf{(U2)} or \textbf{(U4)}.

If condition (2) in \lemref{lem: kmodL} holds with $\mu=1$, then there exists $j\in \{p, \dots ,m-1\}$ such that 
$$
 n\geq \mu_{p+1}\geq \dots \geq \mu_j>\mu_{j+1}= \cdots =\mu_m=\mu-1=0.
$$
It follows that $(\Lambda, \epsilon_{j+1}-\epsilon_m)=0$ 
and that $\Omega_{q|n}$ is a $(\mathfrak{gl}_{p|n}\oplus \mathfrak{gl}_q)$-highest-weight vector of weight 
$$
 \Lambda_{q|n}= \Lambda- \sum_{i>p}^{j}\sum_{\nu=1}^{\mu_i} (\epsilon_i-\delta_{\nu}).
$$
Again using \thmref{thm: type2}, we have either
$$
 (\Lambda_{q|n}+ \rho_{p|n}, \epsilon_1-\delta_1)= (\Lambda, \epsilon_1-\delta_1)+j-1<0,
$$
in which case $\Lambda$ satisfies \textbf{(U5)},
or there exists $i\in \{1, \dots, p\}$ such that 
\begin{align*}
 (\Lambda_{q|n}+\rho_{p|n}, \epsilon_i-\delta_1)= (\Lambda, \epsilon_i-\delta_1)+ j-i=0,
 \qquad
 (\Lambda_{q|n}, \epsilon_1-\epsilon_i)= (\Lambda, \epsilon_1-\epsilon_i)=0,
\end{align*}
in which case $\Lambda$ satisfies \textbf{(U6)}.
\end{proof}

\section{A unitarity criterion}
\label{sec: inv}

Here, we introduce a quadratic invariant of the subalgebra $\mathfrak{k}$
that acts by scalar multiplication on each simple $\mathfrak{k}$-submodule of $L(\Lambda)$ for
$L(\Lambda)$ unitary. This yields a useful criterion for the unitarity of $L(\Lambda)$
that we will use later to prove the sufficiency of the conditions \textbf{(U1)}--\textbf{(U6)}.
We let $\Pi_{\mathfrak{k}}(\Lambda)$ denote the set of all $\mathfrak{k}$-highest weights of $L(\Lambda)$, and
recall the notation $m=p+q$.

We denote by $\rho_{\mathfrak{k}}$ the graded half-sum of positive roots of $\mathfrak{k}$, and note that
$$
 \rho=\rho_{\mathfrak{k}}+\frac{1}{2}\Big(\sum_{i=1}^p\sum_{j=p+1}^m(\epsilon_i-\epsilon_j)
 -\sum_{i=1}^p\sum_{\mu=1}^n(\epsilon_i-\delta_\mu)\Big).
$$
We also note that the Casimir element of $\mathfrak{k}$ is given by
$$
 C_{\mathfrak{k}}=\sum_{i,j=1}^p E_{ij}E_{ji}+\sum_{a,b=p+1}^{m+n} (-1)^{[b]}E_{ab}E_{ba}
$$
and set
$$
 \Gamma:= \sum_{i=1}^{p}\sum_{a=p+1}^{m+n} E_{ai} E_{ia}.
$$
By construction, $\Gamma\in {\rm U}(\mathfrak{g})$, and it is straightforward to verify that $\Gamma$ is $\mathfrak{k}$-invariant
in the sense that $[\Gamma,X]=0$ for all $X\in \mathfrak{k}$. 

Recall that $V(\Lambda)={\rm U}(\mathfrak{k}_{-})\otimes L_{0}(\Lambda)$, where $L_{0}(\Lambda)$ is a simple 
$\mathfrak{k}$-module and ${\rm U}(\mathfrak{k}_{-})$ is isomorphic to the supersymmetric algebra 
$S((\mathbb{C}^{p})^\ast\otimes \mathbb{C}^{q|n})$ as a $\mathfrak{k}$-module. 
Note that $S((\mathbb{C}^{p})^\ast\otimes \mathbb{C}^{q|n})$ is a type-1 unitary $\mathfrak{k}$-module 
and hence $\mathfrak{k}$-semisimple (cf. \cite{CLZ04} and \cite[Theorem 2.1]{Ser01}), 
while $V(\Lambda)$ need not be $\mathfrak{k}$-semisimple. 
\begin{lemma}\label{lem: kinv}
Let $v_{\xi}$ be a $\mathfrak{k}$-highest-weight vector in $V(\Lambda)$ of weight $\xi$. 
Then, $\Gamma$ acts on $v_{\xi}$ as multiplication by the scalar 
$$
 \gamma=\frac{1}{2}(\Lambda-\xi, \Lambda+\xi+2\rho).
$$
\end{lemma}
\begin{proof}
The Casimir element of $\mathfrak{gl}_{m|n}$ is given by
\begin{align*}
 C
 &= \sum_{a,b=1}^{m+n} (-1)^{[b]} E_{ab}E_{ba}
 =C_{\mathfrak{k}}+ \sum_{i=1}^{p}\sum_{a=p+1}^{m+n} (E_{ai}E_{ia} + (-1)^{[a]} E_{ia}E_{ai})
 \\
 &= C_{\mathfrak{k}} + 2\Gamma + \sum_{i=1}^{p}\sum_{a=p+1}^{m+n}(-1)^{[a]} [E_{ia}, E_{ai}],
\end{align*}
and acts on $L(\Lambda)$ as multiplication by the scalar $(\Lambda, \Lambda+2\rho)$, 
while $C_{\mathfrak{k}}$ acts on $v_{\xi}$ as multiplication by the scalar $(\xi, \xi+2\rho_{\mathfrak{k}})$.
As
$$
 [E_{ia}, E_{ai}] v_{\xi}= (E_{ii}- (-1)^{[a]} E_{aa})v_{\xi}= (\xi, \epsilon_i-\epsilon_a) v_{\xi}
$$
for all $i\in\{1,\ldots,p\}$ and all $a\in\{p+1,\ldots,m+n\}$, $\Gamma$ acts on $v_{\xi}$ as multiplication by
\begin{align*}
 \gamma 
 &=\frac{1}{2}\Big((\Lambda, \Lambda+2\rho)-(\xi, \xi+2\rho_{\mathfrak{k}})
 - \sum_{i=1}^{p}\sum_{a=p+1}^{m+n}(-1)^{[a]}(\xi, \epsilon_i-\epsilon_a)\Big)
 \\[.1cm]
 &= \frac{1}{2}( (\Lambda, \Lambda+2\rho)-(\xi, \xi+2\rho))
 = \frac{1}{2}(\Lambda-\xi, \Lambda+\xi+2\rho).
\end{align*} 
\end{proof}

To construct a unitary structure on $V(\Lambda)$, we use \eqref{eq: VLZ} and start with a unitary $\mathfrak{k}$-module $V_0(\Lambda)$
equipped with a fixed positive-definite contravariant Hermitian form $\langle-,-\rangle$. This form can then be extended to a 
contravariant Hermitian form on $V(\Lambda)$ such that \eqref{eq: Hermitian} is satisfied, with $V_k(\Lambda)$ and $V_{\ell}(\Lambda)$ 
orthogonal for $k\neq \ell$; cf. \cite[Lemma 1]{GZ90a}. 

The following result appears in \cite[Lemma 3.1]{GL96} in the context of quantum supergroups. 
\begin{lemma}\label{lem: ortho}
Let $V$ be a $\mathfrak{k}$-module equipped with a contravariant Hermitian form $\langle-, -\rangle$. 
\begin{enumerate}
\item[{\rm (1)}] If $v_{\mu}$ and $v_{\nu}$ are weight vectors of $V$ of weights $\mu\neq\nu$, then $\langle v_{\mu}, v_{\nu}\rangle=0$. 
\vspace{0.15cm}
\item[{\rm (2)}] If $L(\mu)$ and $L(\nu)$ are simple submodules of $V$ of highest weights $\mu\neq\nu$, 
then $\langle L (\mu), L(\nu) \rangle=0$.
\end{enumerate}
\end{lemma}
\begin{proof}
For part (1), since $\mu\neq \nu$, there exists $a\in\{1,\ldots,m+n\}$ such that $\mu(E_{aa})\neq \nu(E_{aa})$. As the star-operation 
fixes the Cartan subalgebra of $\mathfrak{k}$, we have 
$$
 \mu(E_{aa})\,\langle v_{\mu}, v_{\nu}\rangle 
 = \langle E_{aa}v_{\mu}, v_{\nu}\rangle 
 = \langle v_{\mu}, (E_{aa})^{\medstar} v_{\nu}\rangle
 = \nu(E_{aa})\,\langle v_{\mu}, v_{\nu}\rangle,
$$
hence $\langle v_{\mu}, v_{\nu}\rangle=0$. 
For part (2), let $v_{\mu}$ and $v_{\nu}$ be highest-weight vectors of $L(\mu)$ and $L(\nu)$, respectively. 
Let $\mathfrak{b}'= \mathfrak{b}\cap \mathfrak{k}$ be the standard Borel subalgebra of $\mathfrak{k}$, 
and let $\mathfrak{n}'_+$ (respectively $\mathfrak{n}'_-$) be the nilpotent radical (respectively opposite nilpotent radical) 
of~$\mathfrak{b}'$. Since the Cartan elements are fixed under the star-operation and act as scalars on highest-weight vectors,
it follows from $0=\langle v_{\mu}, v_{\nu}\rangle$ that
\begin{align*}
 0 &= \langle \mathrm{U}(\mathfrak{b}') v_{\mu}, v_{\nu}\rangle 
 = \langle v_{\mu},\mathrm{U}(\mathfrak{n}'_{-}) v_{\nu}\rangle
 = \langle v_{\mu}, L(\nu)\rangle 
 = \langle v_{\mu}, \mathrm{U}(\mathfrak{k}) L(\nu)\rangle
 = \langle \mathrm{U}(\mathfrak{k}) v_{\mu},L(\nu)\rangle 
\\[.1cm] 
 & = \langle L (\mu), L(\nu) \rangle.
\end{align*}
\end{proof}

As the simple quotient of $V(\Lambda)$, the module $L(\Lambda)$ inherits the extended contravariant Hermitian form on
$V(\Lambda)$. As demonstrated in the next proposition, a simple criterion ensures that this form on $L(\Lambda)$ is positive-definite. 
\begin{proposition}\label{prop: nscon}
Let $\Lambda\in \Dpqn$, and suppose $L_0(\Lambda)$ is a type-1 unitary simple $\mathfrak{k}$-module. 
Then, $L(\Lambda)$ is unitary if and only if
$$
 (\Lambda-\xi, \Lambda+\xi+2\rho)\leq 0,\qquad
 \forall\,\xi\in \Pi_{\mathfrak{k}}(\Lambda).
$$
\end{proposition}
\begin{proof}
If $L(\Lambda)$ is unitary, then $L(\Lambda)$ is a semisimple $\mathfrak{k}$-module and there exists 
a $\mathfrak{k}$-highest-weight vector $v_{\xi}$ of weight $\xi$ for each $\xi\in \Pi_{\mathfrak{k}}(\Lambda)$.
It follows that
$$
 \langle\Gamma v_{\xi}, v_{\xi} \rangle 
 = - \sum_{i=1}^p \sum_{a=p+1}^{m+n} \langle E_{ia} v_{\xi}, E_{ia} v_{\xi} \rangle \leq 0,
$$
so $(\Lambda-\xi, \Lambda+\xi+2\rho)\leq0$ by \lemref{lem: kinv}.

As to the converse, we use induction on the $\mathbb{Z}$-grading \eqref{eqn: decomp} of $L(\Lambda)$, and let 
$L^{p,q|n}(\xi)\subseteq V(\Lambda)$ denote a simple $\mathfrak{k}$-module 
generated by a highest-weight vector $v_{\xi}$ of weight $\xi\in \Pi_{\mathfrak{k}}(\Lambda)$. 
As $L_{0}(\Lambda)$ is a type-1 unitary $\mathfrak{k}$-module equipped with a positive-definite contravariant Hermitian 
form $\langle-, -\rangle$, $V_k(\Lambda)= S_k((\mathbb{C}^{p})^\ast\otimes \mathbb{C}^{q|n})\otimes L_{0}(\Lambda)$ is a
type-1 unitary $\mathfrak{k}$-module and thus $\mathfrak{k}$-semisimple. It follows that $L_k(\Lambda)$ is a 
semisimple $\mathfrak{k}$-module for any $k\in \mathbb{Z}_+$, so we have a finite $\mathfrak{k}$-module decomposition of the form
$$
 L_k(\Lambda)\cong \bigoplus_{r} K^{p,q|n}(\xi_r),
$$
where $K^{p,q|n}(\xi_r)$ is a direct sum of simple $\mathfrak{k}$-modules isomorphic to $L^{p,q|n}(\xi_r)$.

For the induction step, let $k>1$ and assume that the inherited contravariant Hermitian form $\langle-,- \rangle$ 
is positive-definite on $L_{k-1}(\Lambda)$.
Now, every nonzero vector $v_r\in K^{p,q|n}(\xi_r)$ is a linear combination of vectors of the form 
$E_{ai}w$, where $i\in\{1,\ldots,p\}$, $a\in\{p+1,\ldots,m+n\}$, and $w\in L_{k-1}(\Lambda)$. 
Moreover, there exist $j \in \{1, \dots, p\} $ and $b \in \{p+1, \dots, m+n\}$ such that $E_{jb}v_r \neq 0$; otherwise, we would have a contradiction 
with $v_r\in L_k(\Lambda)$. As $\Gamma$ acts on $v_r$ as scalar multiplication by $\gamma$, the induction hypothesis implies that
$$
 \gamma\, \langle v_r, v_r\rangle =- \sum_{i=1}^p \sum_{a=p+1}^{m+n} \langle E_{ia} v_r, E_{ia} v_r \rangle <0,
$$
and since
$$
 \gamma=\frac{1}{2}(\Lambda-\xi, \Lambda+\xi+2\rho)\leq 0,
$$
we have $\langle v_r, v_r\rangle>0$. 
By \lemref{lem: ortho}, $K^{p,q|n}(\xi_i)$ and $K^{p,q|n}(\xi_j)$ are orthogonal for
$i \neq j$, so for any vector $v=\sum_{r}v_r$ with $v_r\in K^{p,q|n}(\xi_r)$ for each $r$,
$$
 \langle v, v \rangle = \sum_{r}\, \langle v_r , v_r \rangle >0.
$$
It follows that the contravariant Hermitian form $\langle -, -\rangle $ on $L_{k}(\Lambda)$ is positive-definite, 
and since different graded components are orthogonal, $L(\Lambda)$ is unitary.
\end{proof}

\section{Howe duality}
\label{sec: Howe}

In this section, we recall from \cite{CLZ04} (see also \cite{CW01})
the action of $\mathfrak{gl}_{d}\times \mathfrak{gl}_{p+q|n}$ on the supersymmetric 
algebra $S( (\mathbb{C}^{d})^\ast\otimes (\mathbb{C}^{p})^\ast \oplus \mathbb{C}^d\otimes \mathbb{C}^{q|n})$ for any fixed 
positive integer~$d$. This gives rise to the $(\mathfrak{gl}_{d}, \mathfrak{gl}_{p+q|n})$-Howe duality, yielding a multiplicity-free 
decomposition of the supersymmetric algebra into simple $(\mathfrak{gl}_{d}\oplus \mathfrak{gl}_{p+q|n})$-modules. 
In \secref{sec: Integral}, this enables an explicit construction of infinite-dimensional unitary 
$\mathfrak{gl}_{p+q|n}$-modules with integral highest weights as classified in Section \ref{sec: class}.
We recall the notation $m=p+q$.

\subsection{Oscillator superalgebra}
\label{sec: glgl-action}

Fix a positive integer $d$, and let $\mathbb{C}^d$ denote the natural $\mathfrak{gl}_{d}$-module with standard basis $\{v_1,\dots, v_d\}$. 
Let $(\mathbb{C}^{d})^\ast$ be the dual module spanned by the dual basis $\{v^{1}, \dots, v^{d}\}$ such that $v^{a}(v_{b})=\delta_{a,b}$ 
for all $a,b\in\{1,\ldots,d\}$. For each pair $a,b\in\{1,\ldots,d\}$, denote by $e_{ab}$ the matrix unit, 
so $e_{ab}v_c= \delta_{b,c}v_a$ for all $c\in \{1,\dots d\}$.
Also, let $\{e_1,\dots, e_{m+n}\}$ denote the standard homogeneous basis for the 
natural $\mathfrak{gl}_{p+q|n}$-module $\mathbb{C}^{m|n}$, and let $\{e^1, \dots, e^{m+n}\}$ be the basis for
$(\mathbb{C}^{m|n})^\ast$ such that $e^{a}(e_b)=\delta_{a,b}$ for all $a,b\in\{1,\ldots,m+n\}$. 
We identify $(\mathbb{C}^p)^\ast$ with the subspace of $(\mathbb{C}^{m|n})^\ast$ spanned by $\{e^1, \dots, e^p\}$, 
and $\mathbb{C}^{q|n}$ with the subspace of $\mathbb{C}^{m|n}$ spanned by $\{e_{p+1}, \dots, e_{m+n}\}$.

The supersymmetric algebra $S( (\mathbb{C}^{d})^\ast\otimes (\mathbb{C}^{p})^\ast \oplus \mathbb{C}^d\otimes \mathbb{C}^{q|n})$ is 
isomorphic to the polynomial superalgebra $\Cpqnd[{\bf x}, {\bf y}, \eta]$ in the following variables:
\begin{align}\label{eq: xynu}
 x^{a}_k: = v_{a}\otimes e_{p+k}, \qquad 
 y_i^{a}:= v^{a}\otimes e^i, \qquad 
 \eta_{\mu}^{a}:= v_{a}\otimes e_{m+\mu}, 
\end{align}
with
$$
 a\in\{1,\ldots,d\},\qquad
 k\in\{1,\ldots,q\},\qquad
 i\in\{1,\ldots,p\},\qquad
 \mu\in\{1,\ldots,n\}.
$$ 
Note that $x_k^{a}$ and $y_i^{a}$ are even variables, 
while $\eta_{\mu}^{a}$ are odd. Writing $\partial_{ x^{a}_k}, \partial_{y_i^{a}}, \partial_{\eta_{\mu}^{a}}$ for the partial derivatives 
with respect to these variables, we let $\Dpqnd[{\bf x}, {\bf y}, \eta]$ denote the oscillator superalgebra generated by the 
variables $x^{a}_k$, $y^{a}_i$, $\eta^a_{\mu}$ and their derivatives $\partial_{x^a_k}$, $\partial_{y^a_i}$, $\partial_{\eta^a_{\mu}}$. 
Then, $\Cpqnd[{\bf x}, {\bf y}, \eta]$ is a simple module over $\Dpqnd[{\bf x}, {\bf y}, \eta]$.

Let $\rho$ be the associative superalgebra homomorphism
$$
 \rho: {\rm U}(\mathfrak{gl}_{d} \oplus \mathfrak{gl}_{p+q|n}) \to \mathbb{D}[{\bf x}, {\bf y}, \eta],
$$
defined by
$$
 \rho(e_{ab})
 =\sum_{k=1}^{q} x^a_k\, \partial_{x_k^b} 
 - \sum_{i=1}^p y_i^b \partial_{y_i^a}
 + \sum_{\mu=1}^n \eta^{a}_{\mu} \partial_{\eta_{\mu}^b}, \qquad 
 a,b\in\{1,\ldots,d\},
$$
and
\begin{align*}
 &\rho(E_{i,j})= - \sum_{a=1}^d \partial_{y^a_i}y^a_j,&\,& 
 \rho(E_{i,p+\ell}) = \sum_{a=1}^d \partial_{y^a_i} \partial_{x^a_\ell}, &\, &
 \rho (E_{i,p+q+\nu})= \sum_{a=1}^d \partial_{y^a_i} \partial_{\eta^a_\nu}, 
 \\
 &\rho(E_{p+k, j})= - \sum_{a=1}^d x_{k}^a y_j^a, &\,& 
 \rho(E_{p+k,p+\ell}) = \sum_{a=1}^d x_k^a \partial_{x_\ell^a}, &\, & 
 \rho (E_{p+k, p+q+\nu})= \sum_{a=1}^d x_{k}^a \partial_{\eta^a_{\nu}},
 \\
 & \rho(E_{p+q+\mu, j})= - \sum_{a=1}^d \eta_{\mu}^a y_{j}^a, &\, & 
 \rho(E_{p+q+\mu,p+\ell})= \sum_{a=1}^d \eta_{\mu}^a \partial_{x_\ell^a}, &\, & 
 \rho(E_{p+q+\mu,p+q+\nu})= \sum_{a=1}^d \eta_{\mu}^a\partial_{\eta_{\nu}^a}.
\end{align*}
It is straightforward to verify that the differential operators $\rho(e_{ab})$ (respectively $\rho(E_{ij})$) satisfy the commutation relations 
of $\mathfrak{gl}_{d}$ (respectively $\mathfrak{gl}_{p+q|n}$), 
and that $\rho(e_{ab})$ commutes with $\rho(E_{ij})$ for all relevant $a,b,i,j$. 
This gives a realisation of $\mathfrak{gl}_d\times \mathfrak{gl}_{p+q|n}$ in terms of differential operators, 
yielding a linear action on $\Cpqnd[\mathbf{x}, \mathbf{y}, \eta]$. 

The oscillator superalgebra $\mathbb{D}^{p,q|n}_d[{\bf x}, {\bf y}, \eta]$ admits the star-superalgebra structure $\psi$ defined by
$$
 \psi(z)=\partial_z,\qquad
 \psi(\partial_z)=z,
$$
for all variables $z$ given in \eqref{eq: xynu}.
There exists a unique Hermitian form $\langle -, - \rangle$ on $\Cpqnd[{\bf x}, {\bf y}, \eta]$ satisfying $\langle 1, 1 \rangle =1$ and
\begin{align}\label{eq: Herm}
 \langle fg, h\rangle = \langle g, \psi(f)h\rangle, \qquad 
 f,g,h\in \Cpqnd[{\bf x}, {\bf y}, \eta].
\end{align} 
Consequently, $\langle M, M\rangle>0$ for every nonzero monomial $M\in \Cpqnd[{\bf x}, {\bf y}, \eta]$, and $\langle -, - \rangle$ is 
positive-definite. Similarly, ${\rm U}(\mathfrak{gl}_{d}\oplus \mathfrak{gl}_{p+q|n})$ has a star-operation $\sigma$ given by 
$$
 \sigma(e_{ab})=e_{ba}, \qquad
 \sigma(E_{ij})= \begin{cases}
 E_{ji}, & i,j\leq p\ \ \text{or}\ \ i,j>p,\\[.1cm]
 -E_{ji}, & \text{otherwise,} 
 \end{cases}
$$
for all applicable $a,b,i,j$. These two star-structures are compatible in the sense that 
$$
 \rho\sigma(X)= \psi \rho(X), \qquad 
 X\in {\rm U}(\mathfrak{gl}_{d}\oplus \mathfrak{gl}_{p+q|n}),
$$
and we conclude that $\Cpqnd[\mathbf x,\mathbf y,\eta]$ is a unitary 
${\rm U}(\mathfrak{gl}_{d}\oplus \mathfrak{gl}_{p+q|n})$-module
with respect to the Hermitian form \eqref{eq: Herm}; cf. \cite[Theorem~3.2]{CLZ04}.

\subsection{Decomposition of supersymmetric algebra}

A \textit{partition} $\lambda= (\lambda_1, \dots, \lambda_k)$ of length $k$ is a non-increasing sequence of non-negative integers: 
$\lambda_1\geq\cdots\geq\lambda_k\geq0$. We denote by $\mathcal{P}_k$ the set of partitions of length $k$. 
The \textit{conjugate partition} of $\lambda\in\mathcal{P}_k$ is $\lambda'= (\lambda'_1, \dots, \lambda'_{\ell})$, 
where $\ell=\lambda_1$ and $\lambda'_i= \#\{j \,|\, \lambda_j\geq i\}$ for $i=1, \dots, \ell$. 
If $\lambda_1=0$, we set $\lambda'=(0)$. A \textit{generalised partition} of length $k$ is a non-increasing sequence of $k$ integers 
(some of which could be negative). By construction, each generalised partition $\lambda=(\lambda_1, \dots, \lambda_k)$ 
can be written as $\lambda=\lambda_{+}+ \lambda_{-}$, where 
$$
 \lambda_+:= ( {\rm max}\{\lambda_1, 0\}, \dots , {\rm max}\{\lambda_k, 0\} ), \qquad 
 \lambda_{-}:= ( {\rm min}\{\lambda_1, 0\}, \dots , {\rm min}\{\lambda_k, 0\} ). 
$$
We also introduce
$$ 
 \lambda_{-}^*:=(-{\rm min}\{\lambda_k, 0\}, \dots , -{\rm min}\{\lambda_1, 0\}),
$$
and note that $\lambda_{+},\lambda_{-}^*\in\mathcal P_k$. 
We adopt the convention that $\lambda_i=0$ if the index $i$ is non-positive or greater than~$k$.

Denote by $\Ppqnk$ the set of generalised partitions $\lambda=(\lambda_1, \dots, \lambda_k)$ of length $k$ 
such that $\lambda_{q+1}\leq n$ and $\lambda_{k-p}\geq 0$. Associated to a generalised partition 
$\lambda=(\lambda_1, \dots, \lambda_d)\in \Ppqnd$, 
we define a sequence $\lambda^\flat$ of length $p+q+n$ by setting 
\begin{align}\label{eq: laflat}
 \lambda^{\flat}:
 &=(-d,\dots,-d, -d+\lambda_r, \dots, -d+\lambda_d,(\lambda_+)_1, \dots, (\lambda_+)_q,
 \nonumber\\[.1cm]
 &\qquad\qquad\qquad\qquad\qquad\qquad
 {\rm max}\{(\lambda'_+)_1-q,0\}, \dots, {\rm max}\{(\lambda'_+)_n-q\}).
\end{align}
Here, $r\in \{ d-p+1, \dots, d\}$ is the smallest index such that $\lambda_r<0$, if such an index exists; 
otherwise, the first $p$ entries of $\lambda^{\flat}$ are all $-d$. 

To each generalised partition $\lambda\in\Ppqnd$, we associate the unique dominant integral $\mathfrak{gl}_d$-weight
$$
 \lambda= \sum_{i=1}^d\lambda_i \varepsilon_i,
$$
where $\{\varepsilon_1,\dots, \varepsilon_d\}$ is the standard basis for the dual space of the Cartan subalgebra of $\mathfrak{gl}_d$.
Let $L^d(\lambda)$ denote the corresponding simple highest-weight module. Similarly, we identify $\lambda^\flat$ 
with a $\Phi^+_c$-dominant integral weight of $\mathfrak{gl}_{p+q|n}$ via \eqref{def: Lambda}. 
The following theorem is derived from \cite[Lemma~3.2, Theorem 3.3]{CLZ04} 
and establishes the $(\mathfrak{gl}_{d}, \mathfrak{gl}_{p+q|n})$-Howe duality.
\begin{theorem}\label{thm: Howe}
Under the action of $\mathfrak{gl}_{d}\times \mathfrak{gl}_{p+q|n}$ described in Section \ref{sec: glgl-action}, 
$\Cpqnd[{\bf x}, {\bf y}, \eta]$ decomposes into a multiplicity-free direct sum of simple modules, as
$$
 \Cpqnd[{\bf x}, {\bf y}, \eta] 
 \cong \bigoplus_{\lambda\in \Ppqnd} L^d(\lambda)\otimes L^{p+q|n}(\lambda^{\flat}). 
$$
\end{theorem}
\begin{corollary}
Every infinite-dimensional $\mathfrak{gl}_{p+q|n}$-module $L(\lambda^{\flat})$ with $\lambda\in \mathcal{P}_{p+q|n}$ is unitary. 
\end{corollary}

\section{Unitary modules with integral highest weights}
\label{sec: Integral}

Let $\Ppqn$ denote the subset of $\Dpqn$ that consists of the weights whose components are all integers. 
Using Howe duality, we have the following classification result, recalling the notation $m=p+q$.
\begin{theorem}\label{thm: intuni}
Let $\Lambda=(\lambda_1, \dots, \lambda_{p+q}, \omega_1, \dots, \omega_n)\in \Ppqn$. 
Then, $L(\Lambda)$ is unitary 
if and only if one of the following conditions holds:
\begin{enumerate}
\item[{\rm (1)}] Either $\lambda_1+\omega_1<1-m$ or there exists $i\in \{1,\dots, p\}$ such that 
$$
 \lambda_1=\dots =\lambda_i, \qquad 
 \lambda_1+\omega_1= i-m,
$$
and either $\lambda_m+\omega_n>n-1$ or there exists $\mu\in \{2, \dots, n\}$ such that
$$
 \omega_{\mu}=\dots =\omega_{n}, \qquad 
 \lambda_{m}+\omega_n= \mu-1.
$$
\item[{\rm (2)}] There holds
$$
 \omega_1=\dots=\omega_n, \qquad 
 \lambda_m+\omega_{n}=0,
$$
and there exists $j\in \{p, \dots, m-1\}$ such that 
$$
 \lambda_{j+1}=\dots =\lambda_m,
$$
and either $\lambda_1+\omega_1<1-j$ or there exists $i\in \{1,\dots, p\}$ such that 
$$
 \lambda_1=\dots =\lambda_i, \qquad
 \lambda_1+\omega_1= i-j.
$$
\end{enumerate}
\end{theorem}
\begin{proof}
The necessity follows from \propref{prop: nec}, as the conditions \textbf{(U1)}--\textbf{(U6)} can be expressed 
in terms of integral weight components as follows. Using \eqref{eq: rhoform}, we have 
$$
 (\Lambda+\rho, \epsilon_m-\delta_n)= \lambda_m+\omega_n+ 1-n,\qquad
 (\Lambda+\rho, \epsilon_1-\delta_1)= \lambda_1+\omega_1+ m-1.
$$
Hence, condition \textbf{(U1)} is equivalent to the inequalities $\lambda_m+\omega_n>n-1$ and $\lambda_1+\omega_1<1-m$. 
Similarly, condition \textbf{(U2)} requires the existence of $i\in \{1, \dots, p\}$ such that $\lambda_1=\cdots= \lambda_i$ 
and $\lambda_1+\omega_1=i-m$. Combining conditions \textbf{(U1)}--\textbf{(U4)}, we obtain part (1) of the theorem.
Part (2) follows from the combination of \textbf{(U5)} and \textbf{(U6)}. 

To establish sufficiency, we first suppose $\Lambda$ satisfies condition (1) 
and show that $L(\Lambda)$ is, up to tensoring with a 1-dimensional module, 
a submodule of $S( (\mathbb{C}^{d})^\ast\otimes (\mathbb{C}^{p})^\ast \oplus \mathbb{C}^d\otimes \mathbb{C}^{q|n})$ for some 
positive integer~$d$. To this end, we define the generalised partitions 
\begin{align*}
 \lambda^1(\Lambda)&= (\lambda_{p+1}+\omega_n, \dots, \lambda_{m}+\omega_n), 
 \\[.1cm] 
 \lambda^2(\Lambda)&= (\omega_1-\omega_n, \dots, \omega_{\mu-1}-\omega_n)', 
 \\[.1cm] 
 \lambda^3(\Lambda)&= (\lambda_{i+1}-\lambda_1, \dots, \lambda_p-\lambda_1), 
\end{align*}
where the prime denotes conjugation of partition, and we set $\mu:=n$ if $\lambda_{m}+\omega_n>n-1$. 
Here, $\lambda^1(\Lambda)$ and $\lambda^2(\Lambda)$ are partitions, 
while $\lambda^3(\Lambda)$ is a generalised partition of non-positive integers. 
Writing $\ell(S)$ for the length of such a (generalised) partition, we find 
$$
 \sum_{k=1}^3\ell( \lambda^k(\Lambda))\leq m-i{}+\omega_1-\omega_n\leq -(\lambda_1+\omega_1)+\omega_1-\omega_n
 = -\lambda_1-\omega_n,
$$
where the first inequality holds as $i$ is not required to be the maximal index that satisfies $\lambda_1=\dots =\lambda_i$, 
and similarly for $\mu$. Let $d=-\lambda_1-\omega_n$ and define the length-$d$ generalised partition 
$$
 \lambda=(\lambda^1(\Lambda),\lambda^2(\Lambda),0, \dots, 0,\lambda^3(\Lambda)),
$$
where $0$ appears $d-\sum_{i=1}^3\ell( \lambda^i(\Lambda))$ times. Clearly, $\lambda_{q+1}=\mu-1\leq n$ and $\lambda_{d-p}\geq 0$, 
so $\lambda\in \mathcal{P}_{p+q|n}\cap \mathcal{P}_d$. It follows from \thmref{thm: Howe} that $L^{p+q|n}(\lambda^{\flat})$ appears as a 
submodule of $S( (\mathbb{C}^{d})^\ast\otimes (\mathbb{C}^{p})^\ast\oplus \mathbb{C}^d\otimes \mathbb{C}^{q|n})$ 
and is therefore unitary. By \eqref{eq: laflat}, the highest weight $\lambda^{\flat}$ is given by
\begin{align*}
 \lambda^{\flat}&= (\underbrace{-d,\dots, -d}_{i}, -d+\lambda_{i+1}-\lambda_1, \dots,-d+\lambda_{p}-\lambda_1,
 \\[.1cm]
 &\qquad\lambda_{p+1}+\omega_n, \dots, \lambda_{m}+\omega_n,\omega_1-\omega_n, \dots, \omega_{\mu-1}-\omega_n, 
 \underbrace{0, \dots,0}_{n-\mu+1}).
\end{align*}
As $\mathfrak{gl}_{p+q|n}$-modules, 
$$
 L(\Lambda)\cong L^{p+q|n}(\lambda^{\flat})\otimes \mathbb{C}_{d+\lambda_1},
$$
and since both $ L^{p+q|n}(\lambda^{\flat})$ and $\mathbb{C}_{d+\lambda_1}$ are unitary, so is $L(\Lambda)$.

The situation when $\Lambda$ satisfies condition (2) is similar, so we only sketch the proof here. Let
$$
 \lambda^1(\Lambda)=(\lambda_{p+1}+\omega_n, \dots, \lambda_j+\omega_n), \qquad
 \lambda^3(\Lambda)=(\lambda_{i+1}-\lambda_1, \dots, \lambda_p-\lambda_1).
$$
Then,
$$
 \ell(\lambda^1(\Lambda))+\ell(\lambda^3(\Lambda))= j-i \leq -\lambda_1-\omega_1= -\lambda_1-\omega_n.
$$
Let $d=-\lambda_1-\omega_n$, and define a generalised partition $\lambda$ of length $d$ by 
$$
 \lambda= (\lambda^1(\Lambda), \underbrace{0,\dots, 0}_{d-j+i}, \lambda^3(\Lambda)).
$$
Clearly, $\lambda\in \mathcal{P}_{p+q|n}\cap \mathcal{P}_d$, and the associated highest
$\mathfrak{gl}_{p+q|n}$-weight is 
$$
 \lambda^{\flat}
 =(\underbrace{-d, \dots,-d}_i,-d+\lambda_{i+1}-\lambda_1,\dots,-d+\lambda_{p}-\lambda_1,\lambda_{p+1}+\omega_n,\dots,
 \lambda_{j}+\omega_n,\underbrace{0,\dots,0}_{m+n-j}).
$$
It follows that $L(\Lambda)\cong L^{p+q|n}(\lambda^{\flat})\otimes \mathbb{C}_{d+\lambda_1}$,
which again implies that $L(\Lambda)$ is unitary. 
\end{proof}

\begin{remark}
Unitary $\mathfrak{su}(p,q|n)$-modules with integral highest or lowest weights are classified in \cite{FN91} using an oscillator 
representation of an orthosymplectic Lie superalgebra. \thmref{thm: intuni} agrees with \cite[Theorem 5.5]{FN91} up to tensoring 
with a 1-dimensional module.
\end{remark}

\section{Sufficiency}
\label{sec: suff}

We are now in a position to complete the proof of the sufficiency portion of Theorem~\ref{thm: main}.
The sufficiency of condition \textbf{(U1)} is thus established in \propref{prop: type1suff}, 
and of condition \textbf{(U3)} in \propref{peop: type2suff}.
Using \lemref{lem: intper} below, the sufficiency of condition \textbf{(U2)} is established in \propref{prop: U2}, 
and of condition \textbf{(U5)} in \propref{prop: U5}.
Using \thmref{thm: intuni}, the sufficiency of condition \textbf{(U4)}, and independently of \textbf{(U6)}, is established in \propref{prop: U46}. 
We recall our notation $m=p+q$.

For $\Lambda\in \Dpqn$, we write $\Lambda=(\lambda_1, \dots, \lambda_m, \omega_1, \dots, \omega_n)$ and recall
that $L(\Lambda)$ is the unique simple quotient of the highest-weight $\mathfrak{gl}_{m|n}$-module 
$V(\Lambda)$ defined in \eqref{eq: hwmod}.
For $\Lambda\in\Ppqn$ and any real number $s\in [0,1]$, it is convenient to introduce
\begin{align*}
 \Lambda_{(s+)}&:= (\lambda_1, \dots, \lambda_p, \lambda_{p+1}+s, \dots , \lambda_m+s, \omega_1,\dots,\omega_n),
 \\[.1cm]
 \Lambda_{(s-)}&:= (\lambda_1-s, \dots, \lambda_p-s, \lambda_{p+1}, \dots , \lambda_m, \omega_1,\dots,\omega_n).
\end{align*}
\begin{proposition}\label{prop: type1suff}
Let $\Lambda\in \Dpqn$ with
\begin{align}\label{L0L}
 (\Lambda+\rho, \epsilon_1-\delta_1)
 <0
 <(\Lambda+\rho, \epsilon_m-\delta_n).
\end{align}
Then, $L(\Lambda)$ is unitary.
\end{proposition}
\begin{proof}
By \lemref{lem: kmodL}, the $\mathfrak{k}$-module $L_0(\Lambda)$ is type-$1$ unitary, so,
according to \propref{prop: nscon}, it suffices to show that $(\Lambda-\xi, \Lambda+\xi+2\rho)\leq 0$
for all $\xi\in \Pi_{\mathfrak{k}}(\Lambda)$. 
As every $\xi\in \Pi_{\mathfrak{k}}(\Lambda)$ is of the form $\xi= \Lambda-\theta$ with
\begin{align}\label{eq: theta}
 \theta=\sum_{i=1}^p \theta_i,\qquad
 \theta_i=\sum_{k=p+1}^m a_{ik}(\epsilon_i-\epsilon_k) +\sum_{\mu=1}^n b_{i\mu} (\epsilon_i-\delta_{\mu}),
\end{align}
where $a_{ik}\in \mathbb{Z}_+$ and $b_{i\mu}\in \{0,1\}$, it follows that
\begin{align}\label{eq: nscon}
(\Lambda-\xi, \Lambda+\xi+2\rho)=2(\Lambda+\rho, \theta)-(\theta, \theta).
\end{align}
We now turn to estimating $(\Lambda+\rho, \theta)$ and $(\theta, \theta)$.
For each $i\in\{1,\ldots,p\}$, we have
\begin{align*}
 (\Lambda+\rho, \theta_i)
 &=\sum_{k=p+1}^m a_{ik}(\Lambda+\rho,\epsilon_i-\epsilon_k) 
 +\sum_{\mu=1}^n b_{i\mu} (\Lambda+\rho,\epsilon_i-\delta_{\mu}),
 \\[.1cm]
 (\theta_i, \theta_i )
 &=\sum_{k,\ell=p+1}^{m}(a_{ik}a_{i\ell}+\delta_{k\ell}a_{ik}a_{i\ell})
 +\sum_{k=p+1}^m\sum_{\nu=1}^n a_{ik}b_{i\nu}
 +\sum_{\ell=p+1}^m\sum_{\mu=1}^n b_{i\mu} a_{i\ell}
+\sum_{\substack{\mu,\nu=1\\ \mu\neq\nu}}^{n} b_{i\mu}b_{i\nu},
\end{align*}
where
\begin{align*}
 (\Lambda+\rho, \epsilon_i-\epsilon_k)
 &=(\Lambda+\rho, \epsilon_i-\epsilon_1) + (\Lambda+\rho, \epsilon_1-\delta_1) +(\Lambda+\rho, \delta_1-\delta_n)
 \\[.1cm]
 &\qquad+ (\Lambda+\rho, \delta_n-\epsilon_m)+ (\Lambda+\rho, \epsilon_m-\epsilon_k),
 \\[.1cm]
 (\Lambda+\rho,\epsilon_i-\delta_{\mu})
 &=(\Lambda+\rho,\epsilon_i-\epsilon_1) 
 + (\Lambda+\rho,\epsilon_1-\delta_1)
 + (\Lambda+\rho,\delta_1-\delta_{\mu}),
\end{align*}
while for all $i,j\in\{1,\ldots,p\}$ such that $i<j$,
$$
 (\theta_i, \theta_j)
 =\sum_{k=p+1}^{m}a_{ik}a_{jk} - \sum_{\mu=1}^n b_{i\mu}b_{j\mu}\geq- \sum_{\mu=1}^n b_{i\mu}b_{j\mu}.
$$
Using \eqref{eq: rhoform}, \eqref{eq: k-high}, and \eqref{L0L}, we have 
\begin{align*}
 &(\Lambda+\rho, \epsilon_i-\epsilon_k)<0, 
 \\[.1cm]
 &(\Lambda+\rho, \epsilon_i-\delta_{\mu})< (\Lambda+\rho,\epsilon_i-\epsilon_1)=(\Lambda,\epsilon_i-\epsilon_1)-(i-1)\leq -(i-1).
\end{align*}
It follows that
\begin{align}\label{eq: sum1}
 (\Lambda+\rho, \theta)=\sum_{i=1}^{p}(\Lambda+\rho, \theta_i)\leq -\sum_{i=1}^p\sum_{\mu=1}^nb_{i\mu}(i-1). 
\end{align}
As $(\theta_i, \theta_i )$ is seen to be non-negative, we also have
\begin{align}\label{eq: sum2}
 (\theta, \theta)
 = \sum_{i=1}^p\, (\theta_i, \theta_i ) 
 + 2 \sum_{1\leq i<j\leq p} (\theta_i, \theta_j)
 \geq - 2 \sum_{1\leq i<j\leq p} \sum_{\mu=1}^n b_{i\mu}b_{j\mu}. 
\end{align}
By combining \eqref{eq: nscon}, \eqref{eq: sum1}, and \eqref{eq: sum2}, it follows that
\begin{align*}
 (\Lambda-\xi, \Lambda+\xi+2\rho) 
 &\leq 2 \sum_{j=2}^p \sum_{\mu=1}^n b_{j\mu} \bigg(\sum_{i=1}^{j-1}b_{i\mu}-j+1\bigg).
\end{align*}
As $b_{i\mu}\in \{0,1\}$, the expression in brackets is non-positive for $j\geq 2$,
so $(\Lambda-\xi, \Lambda+\xi+2\rho)\leq0$ and the unitarity of $L(\Lambda)$ follows from \propref{prop: nscon}.
\end{proof}
\begin{proposition}\label{peop: type2suff}
Let $\Lambda\in \Dpqn$ with $(\Lambda+\rho, \epsilon_1-\delta_1)<0$ and
$$
 (\Lambda+\rho, \epsilon_m-\delta_\mu )=(\Lambda, \delta_\mu-\delta_n)=0
$$
for some $\mu\in \{2,\dots, n\}$. Then, $L(\Lambda)$ is unitary.
\end{proposition}
\begin{proof}
The proof is the same as that of \propref{prop: type1suff}, except the justification for $(\Lambda+\rho, \epsilon_i-\epsilon_k)<0$. 
We now have
$$
 (\Lambda+\rho, \epsilon_i-\epsilon_k)
 =(\Lambda+\rho, \epsilon_i-\epsilon_1)
 +(\Lambda+\rho, \epsilon_1-\delta_1)
 +(\Lambda+\rho, \delta_1-\epsilon_m)
 +(\Lambda+\rho, \epsilon_m-\epsilon_k).
$$
By assumption, $(\Lambda+\rho, \epsilon_1-\delta_1)<0$, and by \eqref{eq: rhoform} and \eqref{eq: k-high},
$$
 (\Lambda+\rho, \epsilon_i-\epsilon_1)\leq 0, \qquad (\Lambda+\rho, \epsilon_m-\epsilon_k)\leq 0.
$$
The condition $(\Lambda+\rho, \epsilon_m-\delta_{\mu})=0$ is equivalent to $\lambda_m+\omega_{\mu}=\mu-1$. Consequently,
$$
 (\Lambda+\rho, \delta_1-\epsilon_m)= -\omega_1-\lambda_m= -\omega_1-(\mu-1-\omega_n)< 0,
$$
so $(\Lambda+\rho, \epsilon_i-\epsilon_k)<0$.
\end{proof}
\begin{lemma}\label{lem: intper}
Let $\Lambda\in \Ppqn$.
\begin{enumerate}
\item[{\rm (1)}] If $\lambda_m+\omega_n\geq n-1$ and there exists $i\in \{1,\dots, p\}$ such that 
$$
 \lambda_1=\dots =\lambda_i, \qquad \lambda_1+\omega_1= i-m,
$$
then $L(\Lambda_{(s+)})$ is unitary for every $s\in[0,1]$.
\vspace{0.15cm}
\item[{\rm (2)}] If $\omega_1=\dots=\omega_n$, $\lambda_m+\omega_{n}=0$, and there exists $j\in \{p, \dots, m-1\}$ such that 
$$
 \lambda_{j+1}=\dots =\lambda_m, \qquad \lambda_1+\omega_1\leq 1-j,
$$
then $L(\Lambda_{(s-)})$ is unitary for every $s\in[0,1]$.
\end{enumerate} 
\end{lemma}
\begin{proof}
For convenience, let $\Lambda_{(s)}$ denote $\Lambda_{(s+)}$ or $\Lambda_{(s-)}$. Recall from \eqref{eq: hwmod} that 
$$
 V(\Lambda_{(s)})= {\rm U}(\mathfrak{k}_-) \otimes V_0(\Lambda_{(s)}),
$$
and that its unique simple quotient is denoted by $L(\Lambda_{(s)})$.
By \lemref{lem: kmodL}, the $\mathfrak{k}$-module $V_0(\Lambda_{(s)})$ is type-1 unitary for every $s \in [0,1]$. 
Moreover, as a $\mathfrak{k}$-module, ${\rm U}(\mathfrak{k}_-) \cong S\bigl((\mathbb{C}^{p})^\ast \otimes \mathbb{C}^{q|n}\bigr)$,
and the latter is also type-1 unitary; see \cite{CLZ04} or \cite[Theorem 2.1]{Ser01}. 
It follows that $V(\Lambda_{(s)})$ is a type-1 unitary $\mathfrak{k}$-module, and hence is $\mathfrak{k}$-semisimple.

According to \lemref{lem: kinv}, the $\mathfrak{k}$-invariant $\Gamma$ acts on 
each simple $\mathfrak{k}$-submodule $L^{p,q|n}(\xi_{(s)})$ of $V(\Lambda_{(s)})$ as multiplication by the scalar 
$$
 \gamma_s=\frac{1}{2}(\Lambda_{(s)}-\xi_{(s)}, \Lambda_{(s)}+ \xi_{(s)}+2\rho).
$$
As $\xi_{(s)}= \Lambda_{(s)}-\theta$, where $\theta$ is of the form \eqref{eq: theta}, it follows that 
$$
 \gamma_s=(\Lambda_{(s)}+\rho,\theta)-(\theta, \theta),
$$ 
so $\gamma_s$ is an affine linear function of~$s$; that is, $\gamma_s= as+b$ for some constants $a,b\in \mathbb{R}$. If $s$ equals 0 or~1, 
then $\Lambda_{(s)}$ is an integral highest weight and \thmref{thm: intuni} implies that $L(\Lambda_{(s)})$ is unitary. 
In this case, by \propref{prop: nscon}, $\gamma_s\leq 0$. 
Since $\gamma_s$ depends linearly on~$s$, this means that $\gamma_s\leq 0$ for all $s\in [0,1]$. 
Another application of \propref{prop: nscon} yields the unitarity of $L(\Lambda_{(s)})$ for all $s\in [0,1]$. 
\end{proof}
\begin{proposition}\label{prop: U2}
Let $\Lambda\in \Dpqn$ with $(\Lambda+\rho, \epsilon_m-\delta_n)>0$, and suppose
there exists $i\in \{1,\dots,p\}$ such that
$$
 (\Lambda+\rho, \epsilon_i-\delta_1)=(\Lambda, \epsilon_i-\epsilon_1)=0.
$$
Then, $L(\Lambda)$ is unitary. 
\end{proposition}
\begin{proof}
The assumptions imply that $\lambda_m+\omega_n>n-1$ and that 
there exists $i\in \{1, \dots, p\}$ such that $\lambda_1=\dots=\lambda_i$ and $\lambda_i+\omega_1=i-m$. 
Clearly, $\omega_{\mu}-\omega_n\in \mathbb{Z}_+$ for $\mu\in\{1,\ldots,n-1\}$. Also, for each $k\in\{1,\dots, p\}$, 
$$
 \lambda_k+\omega_n=(\lambda_k-\lambda_i)+ (\lambda_i+\omega_1)+ (\omega_1-\omega_n)\in \mathbb{Z}_-,
$$
while for each $k\in\{p+1, \dots, m\}$, 
$$
 \lambda_k+\omega_n= \lambda_k-\lambda_m + \lambda+\omega_n> n-1.
$$
Let 
\begin{align}\label{Yn}
 \Upsilon:=(\lambda_1+\omega_n, \dots, \lambda_m+\omega_n, \omega_1-\omega_n, \dots, \omega_{n-1}-\omega_n,0),
\end{align}
and set $s=\lambda_n+\omega_n-\lfloor{\lambda_n+\omega_n}\rfloor$. Then,
$$
 \widetilde{\Upsilon}:= (\lambda_1+\omega_n, \dots, \lambda_p+\omega_n, \lambda_{p+1}+\omega_n-s, \dots,\lambda_{m}+\omega_n-s, 
 \omega_1-\omega_n, \dots, \omega_{n-1}-\omega_n,0)
$$ 
is an integral weight satisfying $\widetilde{\Upsilon}_{(s+)}=\Upsilon$,
so by \lemref{lem: intper}, the simple $\mathfrak{gl}_{m|n}$-module $L(\Upsilon)$ is unitary. 
As $L(\Lambda)\cong L(\Upsilon)\otimes \mathbb{C}_{-\omega_n}$, it follows that $L(\Lambda)$ is unitary.
\end{proof}
\begin{proposition}\label{prop: U5}
Let $\Lambda\in \Dpqn$ with
$$
 (\Lambda, \epsilon_1-\delta_1)<1-j,\qquad
 (\Lambda, \epsilon_{j+1}-\epsilon_m)=(\Lambda+\rho, \epsilon_m-\delta_1)=(\Lambda, \delta_1-\delta_n)= 0,
$$
for some $j\in\{p,\dots, m-1\}$.
Then, $L(\Lambda)$ is unitary.
\end{proposition}
\begin{proof}
For each $k\in\{1,\ldots,p\}$,
$$
 \lambda_k+\omega_1= \lambda_k-\lambda_1 +\lambda_1+\omega_1 <1-j
$$
while for each $k\in\{p+1,\ldots,j\}$, 
$$
 \lambda_k+\omega_1= \lambda_k-\lambda_m+\lambda_m+\omega_1=\lambda_k-\lambda_m \in \mathbb{Z}_+.
$$
Let 
\begin{align}\label{Y1}
 \Upsilon:=(\lambda_1+\omega_1, \dots, \lambda_{j}+\omega_1, 0, \dots, 0),
\end{align}
and set $s=\lfloor{\lambda_1+\omega_1}\rfloor- (\lambda_1+\omega_1)$. Then,
$$
 \widetilde{\Upsilon}:= 
 (\lambda_1+\omega_1+s,\dots, \lambda_p+\omega_1+s, \lambda_s+ \omega_1, \dots, \lambda_j+\omega_1, 0,\dots, 0)
$$
is an integral weight satisfying $\widetilde{\Upsilon}_{(s-)}=\Upsilon$, so by \lemref{lem: intper}, 
the simple $\mathfrak{gl}_{m|n}$-module $L(\Upsilon)$ is unitary. 
As $L(\Lambda)\cong L(\Upsilon)\otimes \mathbb{C}_{-\omega_1}$, it follows that $L(\Lambda)$ is unitary.
\end{proof}
\begin{proposition}\label{prop: U46}
Let $\Lambda\in \Dpqn$, and
suppose there exist $\mu\in \{2,\dots, n\}$ and $i\in \{1,\dots,p\}$ such that
$$
 (\Lambda+\rho, \epsilon_m-\delta_\mu )
 =(\Lambda, \delta_\mu-\delta_n)
 =(\Lambda+\rho, \epsilon_i-\delta_1)
 =(\Lambda, \epsilon_i-\epsilon_1)
 =0,
$$
or there exist $j\in\{p,\dots, m-1\}$ and $i\in \{1,\dots,p\}$ such that 
\begin{align*}
 (\Lambda, \epsilon_{j+1}-\epsilon_m)
 =(\Lambda+\rho, \epsilon_m-\delta_1)
 =(\Lambda, \delta_1-\delta_n)
 =(\Lambda, \epsilon_i-\epsilon_1)
 =0, 
 \quad\
 (\Lambda, \epsilon_i-\delta_1)=i-j.
\end{align*}
Then, $L(\Lambda)$ is unitary. 
\end{proposition}
\begin{proof}
If $\Lambda$ satisfies the first chain of equalities in the proposition, then there exists $\mu \in \{2, \dots, n\}$ 
such that $\omega_{\mu} = \dots = \omega_n$ and $\lambda_m + \omega_n = \mu - 1$. 
Additionally, there exists $i \in \{1, \dots, p\}$ such that $\lambda_1 = \dots = \lambda_i$ and $\lambda_1 + \omega_1 = i - m$. 
It is straightforward to verify that $\Upsilon$ in \eqref{Yn} lies in $\Ppqn$ and satisfies condition~(1) of \thmref{thm: intuni}. 
Consequently, the module $L(\Upsilon)$ is unitary. 
As $L(\Lambda) \cong L(\Upsilon) \otimes \mathbb{C}_{-\omega_n}$, it follows that $L(\Lambda)$ is unitary.

Similarly, the weight $\Upsilon$ in \eqref{Y1} lies in $\Ppqn$ and satisfies condition (2) of \thmref{thm: intuni},
so $L(\Upsilon)$ is unitary, and as $L(\Lambda)\cong L(\Upsilon)\otimes \mathbb{C}_{-\omega_1}$, it follows that $L(\Lambda)$ is unitary.
\end{proof}

\section{Further classifications of unitary modules}
\label{sec: lowest}

Using the classification of unitary $\mathfrak{gl}_{p+q|n}$-modules in \thmref{thm: main}, we now classify the
$\mathfrak{gl}_{p+q|n}$-modules that are unitary with respect to the dual star-operation.
We also present a classification of unitary $\mathfrak{gl}_{n|q+p}$-modules, and recall our notation $m=p+q$.

\subsection{Dual-unitary modules}

Dual to the $p,q$-dependent star-operation $\biggstar$ defined in \eqref{eq: star}, we have the star-operation $\dualstar$ defined by 
$$
 (E_{ab})^{\dualstar}:=(-1)^{[a]+[b]}(E_{ab})^{\medstar}, \qquad a,b\in\{1,\ldots,m+n\}.
$$
A $\mathfrak{gl}_{p+q|n}$-module is \textit{dual-unitary} if it carries a positive-definite Hermitian form satisfying \eqref{eq: Hermitian} 
with $\biggstar$ replaced with $\dualstar$. 

In preparation for our classification of dual-unitary $\mathfrak{gl}_{p+q|n}$-modules, we recall
the triangular decomposition \eqref{triangular}, and introduce
$$
 D_{p+q|n}^-:=-D_{p+q|n}^+.
$$
Mimicking \eqref{Kac}, for each $\Lambda\in D_{p+q|n}^-$, we define 
$$
 K^-(\Lambda):= {\rm U}(\mathfrak{gl}_{m|n}) \otimes_{\,{\rm U}(\mathfrak{g}_{0}\,\oplus\,\mathfrak{g}_{-1}\!)} \!L_{0}(-\Lambda).
$$
The corresponding simple \textit{lowest-weight module} $L^-(\Lambda)$ arises
as the quotient of $K^-(\Lambda)$ by its unique maximal proper submodule.

We now focus on \textit{admissible} $\mathfrak{gl}_{p+q|n}$-modules, 
i.e., those whose restriction to $\mathfrak{gl}_p\oplus \mathfrak{gl}_{q|n}$ decomposes into a 
direct sum of finite-dimensional simple modules with finite multiplicities. All such modules are weight modules,
and we have the following analogues of \lemref{lem: wtcon} and \propref{prop: highest} (the proofs are similar).
\begin{lemma}\label{lem: wtcon2}
Let $V$ be an admissible dual-unitary $\mathfrak{gl}_{p+q|n}$-module, 
and let $\lambda=\sum_{i=1}^{p+q}\lambda_i \epsilon_i + \sum_{\mu=1}^n \omega_{\mu}\delta_{\mu}$ be any weight of $V$. 
Then, $\lambda$ is real, and
$$
 \lambda_j \leq -\omega_{\mu}\leq \lambda_i
$$
for all $i\in\{1,\ldots,p\}$, $j\in\{p+1,\ldots,p+q\}$, and $\mu\in\{1,\ldots,n\}$. 
\end{lemma}  
\begin{proposition}\label{prop: lowest}
Let $V$ be an admissible dual-unitary simple $\mathfrak{gl}_{p+q|n}$-module. Then, $V$ is a lowest-weight module with lowest 
weight $\Lambda=\sum_{i=1}^{p+q}\lambda_i \epsilon_i + \sum_{\mu=1}^n \omega_{\mu}\delta_{\mu}$ satisfying 
$$
 \lambda_{p+1}\leq \cdots \leq \lambda_m \leq -\omega_n \leq \cdots \leq -\omega_1\leq \lambda_1 \leq \cdots \leq \lambda_p.
$$
\end{proposition}
We recall our convention that the graded dual to the $\mathfrak{gl}_{p+q|n}$-module
$L(\Lambda)=\bigoplus_{k\in \mathbb{Z}_+}L_k(\Lambda)$ is defined by 
$$
 L^{\ast}(\Lambda):= \bigoplus_{k\in \mathbb{Z}^+} L_k^{\ast}(\Lambda).
$$
Similar to the finite-dimensional case (see \cite{GZ90a}), we have the following result. 
\begin{proposition}\label{prop: dual}
Let $\Lambda \in D_{p+q|n}^+$. Then, the simple $\mathfrak{gl}_{p+q|n}$-module $L(\Lambda)$ is unitary 
if and only if $L^*(\Lambda)$ is dual-unitary and $L^*(\Lambda)\cong L^-(-\Lambda)$.
\end{proposition}
\begin{proof}
Since taking duals negates the weights, the lowest weight of $L^{*}(\Lambda)$ is $-\Lambda$. 
The result now follows from \propref{prop: dualstar}.
\end{proof}
The following classification is a consequence of  \thmref{thm: main} and \propref{prop: dual}.
\begin{theorem}
Let $\Lambda \in D_{p+q|n}^-$. Then, the simple $\mathfrak{gl}_{p+q|n}$-module $L^-(\Lambda)$ is dual-unitary 
if and only if $-\Lambda$ satisfies one of the conditions \emph{\textbf{(U1)}--\textbf{(U6)}}.
\end{theorem}

\subsection{Unitary modules over \texorpdfstring{\(\mathfrak{gl}_{n\mid q+p}\)}{gl(n|q+p)}}

Related to $\mathfrak{gl}_{n|q+p}$, we introduce
$$
 \tilde{\mathbf I}_{n|q+p}:=\tilde{\mathbf I}_n\cup \tilde{\mathbf I}_q\cup \tilde{\mathbf I}_p,
$$
where
$$
\tilde{\mathbf I}_n:=\{1,\dots,n\}, \qquad
\tilde{\mathbf I}_q:=\{n+1,\dots,n+q\}, \qquad
\tilde{\mathbf I}_p:=\{n+q+1,\dots,m+n\},
$$
noting that the indices in $\tilde{\mathbf I}_n$ are even, while those in $\tilde{\mathbf I}_q$ and $\tilde{\mathbf I}_p$ are odd. 
For $a,b\in \tilde{\mathbf I}_{n|q+p}$, we let $\widetilde E_{ab}$ denote the matrix unit of $\mathfrak{gl}_{n|q+p}$, 
so $[\widetilde E_{ab}]=[a]+[b]$. We also set
$$
\tilde a:=m+n+1-a, \qquad a\in \tilde{\mathbf I}_{n|q+p}.
$$
The standard Borel subalgebra of $\mathfrak{gl}_{n|q+p}$ is spanned by the elements $\widetilde E_{ab}$ with $a\le b$, 
and its Cartan subalgebra is spanned by the diagonal elements $\widetilde E_{aa}$, where $a\in \tilde{\mathbf I}_{n|q+p}$. 
Note that there is an even isomorphism of Lie superalgebras
$$
 \tau:\mathfrak{gl}_{n|q+p}\to \mathfrak{gl}_{p+q|n},\qquad
 \tau(\widetilde E_{ab})=E_{\tilde a,\tilde b},\qquad 
 a,b\in \tilde{\mathbf I}_{n|q+p}.
$$

We also define a sign function $\tilde s:\tilde{\mathbf I}_{n|q+p}\to\{\pm1\}$ by
$$
\tilde s(a):=
\begin{cases}
1, & a\in \tilde{\mathbf I}_n\cup \tilde{\mathbf I}_q,\\[0.1cm]
-1, & a\in \tilde{\mathbf I}_p,
\end{cases}
$$
and use this to define the $p,q$-dependent star-operation $\biggstar$ on $\mathfrak{gl}_{n|q+p}$ by
$$
 (\widetilde E_{ab})^{\medstar}
 :=\tilde s(a)\tilde s(b)\widetilde E_{ba}, \qquad 
 a,b\in \tilde{\mathbf I}_{n|q+p}.
$$
The corresponding dual star-operation $\dualstar$ is then given by
$$
 (\widetilde E_{ab})^{\dualstar}
 :=(-1)^{[a]+[b]}(\widetilde E_{ab})^{\medstar},\qquad 
 a,b\in \tilde{\mathbf I}_{n|q+p}.
$$
\begin{lemma}\label{lem: isostar}
The isomorphism $\tau\colon \mathfrak{gl}_{n|q+p}\to \mathfrak{gl}_{p+q|n}$ preserves the star-operations in the sense that,
for all $a,b\in \tilde{\bf I}_{n|q+p}$,
$$
 \tau\bigl((\widetilde{E}_{ab})^{\medstar}\bigr)=\tau(\widetilde{E}_{ab})^{\medstar}, \qquad 
 \tau\bigl((\widetilde{E}_{ab})^{\dualstar}\bigr)=\tau(\widetilde{E}_{ab})^{\dualstar}.
$$
\end{lemma}
\begin{proof}
For the star-operation $\biggstar$, we have
$$
 \tau((\widetilde{E}_{ab})^{\medstar})= \tilde{s}(a)\tilde{s}(b) E_{\tilde{b}, \tilde{a}}= \tau(\widetilde{E}_{ab})^{\medstar}. 
$$
The relation for the dual star-operation follows similarly, using that $[\tilde{a}]= [a]+\bar{1}$. 
\end{proof}
We are only concerned with \textit{admissible} $\mathfrak{gl}_{n|q+p}$-modules, 
i.e., those whose restriction to $\mathfrak{gl}_{n|q}\oplus \mathfrak{gl}_p$ 
decomposes into a direct sum of finite-dimensional simple modules with finite multiplicities.
For such modules, it is straightforward to establish the following result.
\begin{proposition}\label{prop: lowest2}
Let $V$ be an admissible simple $\mathfrak{gl}_{n|q+p}$-module. 
\begin{enumerate}
\item[{\rm (1)}] If $V$ is  unitary, then $V$ is a lowest-weight module with lowest weight 
$\Lambda=\sum_{\mu=1}^n\omega_{\mu}\epsilon_{\mu} +\sum_{i=1}^{m}\lambda_i\delta_i$ satisfying 
$$
 \lambda_{q+1}\leq \cdots \leq \lambda_m \leq -\omega_n \leq \cdots \leq -\omega_1\leq \lambda_1 \leq \cdots \leq \lambda_q. 
$$
\item[{\rm (2)}] If $V$ is dual-unitary, then $V$ is a highest-weight module with highest weight 
$\Lambda=\sum_{\mu=1}^n\omega_{\mu}\epsilon_{\mu} +\sum_{i=1}^{m}\lambda_i\delta_i$ satisfying 
$$
  \lambda_{q}\leq \cdots \leq \lambda_1 \leq -\omega_1 \leq \cdots \leq -\omega_n\leq \lambda_m \leq \cdots \leq \lambda_{q+1}. 
$$
\end{enumerate}
\end{proposition}
As a consequence of \lemref{lem: isostar}, there is a bijective correspondence between the unitary modules over $\mathfrak{gl}_{n|q+p}$ 
and those over $\mathfrak{gl}_{p+q|n}$. More precisely, let $L(\Lambda)$ be a unitary simple $\mathfrak{gl}_{p+q|n}$-module with 
highest weight $\Lambda=(\lambda_1,\dots, \lambda_p, \lambda_{p+1}, \dots, \lambda_m, \omega_1, \dots, \omega_n)$, 
and let $\pi: \mathfrak{gl}_{p+q|n}\to\mathrm{End}_{\mathbb{C}}(L(\Lambda))$ denote the corresponding linear representation. 
Then, when viewed as a $\mathfrak{gl}_{n|q+p}$-module via $\pi\circ \tau$, $L(\Lambda)$ is a 
unitary simple $\mathfrak{gl}_{n|q+p}$-module with lowest weight 
$$
 \Lambda^{\tau}:=(\omega_n, \dots, \omega_1, \lambda_m, \dots, \lambda_{p+1}, \lambda_p, \dots, \lambda_1). 
$$
Indeed, it is readily verified that $\Lambda^{\tau}$ satisfies the lowest-weight conditions in part (1) of \propref{prop: lowest2}.  
Exchanging the even and odd indices thus interchanges highest-weight unitary $\mathfrak{gl}_{p+q|n}$-modules and 
lowest-weight unitary $\mathfrak{gl}_{n|q+p}$-modules. A similar correspondence holds for the dual-unitary modules.
\begin{theorem}
The simple lowest-weight $\mathfrak{gl}_{n|q+p}$-module $L^-(\Lambda)$ is unitary if and only if $\Lambda^{\tau}$
satisfies one of the conditions \emph{\textbf{(U1)}--\textbf{(U6)}}.
\end{theorem}
Since taking the dual interchanges unitary and dual-unitary $\mathfrak{gl}_{n|q+p}$-modules,
we similarly have the following classification of dual-unitary simple $\mathfrak{gl}_{n|q+p}$-modules. 
\begin{theorem}
The simple highest-weight $\mathfrak{gl}_{n|q+p}$-module $L(\Lambda)$ is dual-unitary if and only if 
$-\Lambda^{\tau}$ satisfies one of the conditions \emph{\textbf{(U1)}--\textbf{(U6)}}.
\end{theorem}

\subsection{Unitary modules over \texorpdfstring{\(\mathfrak{gl}_{p+q|r+s}\)}{gl(p+q|r+s)}}

Let $m=p+q$ and $n=r+s$. It is well-known that, if $pq\neq 0$ and $rs\neq 0$, then the real form $\mathfrak{su}(p,q|r,s)$ admits only 
trivial unitary simple modules; see \cite[Proposition 1]{GV19} and \cite[Lemma 2.1]{FN91}. 
For $\mathfrak{u}(p,q| r,s)$ with $pq\neq0$ and $rs\neq0$, the only unitary simple modules are 1-dimensional; see \propref{prop:pq} below.

Indeed, let
$$
 {\bf I}_{p+q|r+s}:={\bf I}_p\cup {\bf I}_q\cup {\bf I}_r\cup \mathbf{I}_s,
$$
where
\begin{align*}
 &{\bf I}_p:= \{1, \dots, p\}, \qquad\qquad\qquad\, {\bf I}_q:= \{p+1, \dots, p+q\}, 
 \\[.1cm] 
 &{\bf I}_r:= \{m+1, \dots, m+r\}, \qquad {\bf I}_s:= \{m+r+1, \dots, m+n\},
\end{align*}
noting that ${\bf I}_p$ and ${\bf I}_q$ are even, while ${\bf I}_r$ and $\mathbf{I}_s$ are odd. 
We also define the function $s: {\bf I}_{p+q|r+s}\to\{\pm 1\}$ by 
$$
 s(a):= \begin{cases}
  1, & a\in {\bf I}_p\cup \mathbf{I}_s, \\[.1cm]
  -1, & a\in {\bf I}_q\cup {\bf I}_r,
 \end{cases}
$$
and let the star-operation $\biggstar$ and its dual star-operation $\dualstar$ 
on $\mathfrak{gl}_{p+q|r+s}$ act as
$$
 (E_{ab})^{\medstar}:= s(a)s(b) E_{ba},\qquad
 (E_{ab})^{\dualstar}:=(-1)^{[a]+[b]} s(a)s(b) E_{ba},
$$
where $a,b\in {\bf I}_{p+q|r+s}$.
Accordingly, we have unitary $\mathfrak{gl}_{p+q|r+s}$-modules and dual-unitary $\mathfrak{gl}_{p+q|r+s}$-modules. 
Any such module is said to be \textit{admissible} if its restriction to the subalgebra 
$\mathfrak{gl}_{p}\oplus \mathfrak{gl}_{q|r}\oplus \mathfrak{gl}_s$ decomposes into a direct sum of finite-dimensional simple modules 
with finite multiplicities. All such modules are weight modules.
\begin{proposition}\label{prop:pq}
If $pq\neq 0$ and $rs\neq 0$, any admissible unitary or dual-unitary simple $\mathfrak{gl}_{p+q|r+s}$-module is $1$-dimensional. 
\end{proposition}
\begin{proof}
Let $V$ be an admissible unitary simple $\mathfrak{gl}_{p+q|r+s}$-module, and let $v\in V$ be a weight vector with weight
$$
 \lambda= \sum_{i=1}^{p}\lambda_i\epsilon_i+ \sum_{j=p+1}^{p+q}\lambda_j\epsilon_j + \sum_{\mu=1}^r \omega_{\mu}\delta_{\mu} 
  + \sum_{\nu=r+1}^{r+s}\omega_{\nu}\delta_{\nu}. 
$$
For $i\in\{1,\ldots,p\}$ and $\mu\in\{1,\ldots,r\}$, we have 
\begin{align*}
 (\lambda_i + \omega_{\mu})\langle v, v\rangle
 &= \langle [E_{i,m+\mu}, E_{m+\mu, i}]v, v\rangle
 = \langle E_{i,m+\mu}E_{m+\mu, i}v, v\rangle +\langle E_{m+\mu, i}E_{i,m+\mu}v, v\rangle 
 \\[.1cm]
 &= -\langle E_{m+\mu, i}v, E_{m+\mu, i}v\rangle -\langle E_{i,m+\mu}v, E_{i,m+\mu}v\rangle\leq 0. 
\end{align*}
It follows that $\lambda_i+\omega_{\mu}\leq 0$ for all relevant  $i, \mu$. Similarly, we have
$$
  \lambda_i+\omega_{\nu}\geq 0, \qquad 
  \lambda_j+\omega_{\mu}\geq 0, \qquad 
  \lambda_j+\omega_{\nu}\leq 0.
$$
Combining these inequalities, we obtain
$$
 \lambda_i\leq -\omega_{\mu}\leq \lambda_j\leq -\omega_{\nu}\leq \lambda_i,
$$
forcing the inequalities to be equalities, so all weight spaces of $V$ are $1$-dimensional. 
Hence, as $V$ is simple, it is $1$-dimensional. The same holds for the dual-unitary simple modules. 
\end{proof}

\end{document}